\documentclass[a4paper,10pt,reqno, oneside]{amsart}
\usepackage[utf8]{inputenc}
\usepackage{amsmath}
\usepackage{amsthm}
\usepackage{amssymb}
\usepackage{geometry}
\usepackage{enumitem}
\usepackage{nicefrac}
\usepackage{csquotes}
\usepackage{xcolor}
\usepackage{environ}
\usepackage{graphicx}
\usepackage{hyperref}
\usepackage[noabbrev]{cleveref}

\newcommand{\ZZ}{\mathbb{Z}}

\newcommand{\RR}{\mathbb{R}}
\newcommand{\CC}{\mathbb{C}}

\newcommand{\id}{\mathrm{id}}
\newcommand{\I}{\mathrm{i}}

\newcommand{\ad}{\mathrm{ad}}
\newcommand{\Ad}{\mathrm{Ad}}
\newcommand{\liealg}[1]{\mathfrak{#1}}
\newcommand{\lieg}{\liealg{g}}
\newcommand{\liep}{\liealg{p}}
\newcommand{\liek}{\liealg{k}}
\newcommand{\lieh}{\liealg{h}}
\newcommand{\lies}{\liealg{s}}
\newcommand{\liet}{\liealg{t}}
\newcommand{\rank}{\operatorname{rank}}
\newcommand{\Aut}[1][]{\operatorname{Aut}\!#1}
\newcommand{\Out}{\operatorname{Out}}

\theoremstyle{definition}

\theoremstyle{theorem}
\newtheorem{theorem}{Theorem}[section]
\newtheorem{lemma}[theorem]{Lemma}
\newtheorem{proposition}[theorem]{Proposition}
\newtheorem{definition}[theorem]{Definition}
\newtheorem{corollary}[theorem]{Corollary}

\newcounter{claimcounter}[theorem]
\newtheorem{claimenv}{Claim}

\NewEnviron{claim}{%
	\setcounter{claimenv}{\value{claimcounter}}
	\refstepcounter{claimcounter}%
	\begin{claimenv}%
		\BODY%
	\end{claimenv}%
}

\NewEnviron{proofofclaim}[1][\hfill$\blacksquare$]{%
	\noindent\BODY#1%
	\par\medskip
}

\begin{document}
	\title{On the equivariant cohomology of $\ZZ_2\times\ZZ_k$-symmetric spaces}
	\author{Sam Hagh Shenas Noshari}
	\address{Departement of Mathematics, University of Fribourg, Switzerland}
	\email{sam.haghshenasnoshari@unifr.ch}

	\begin{abstract}
		In \cite{Lutz} the class of $\Gamma$-symmetric spaces was introduced, a vast generalization of symmetric spaces. Previous results
		make it conceivable that their isotropy action is equivariantly formal, and we provide evidence for this in case that $\Gamma = \ZZ_2\times\ZZ_k$. 
		This in particular implies that $\ZZ_2\times\ZZ_k$-symmetric spaces are formal in the sense of Rational Homotopy Theory.  
	\end{abstract}

	\maketitle
	
	\section{Introduction}
	
	An action of a Lie group $G$ on a topological space is (rationally) \emph{equivariantly formal} if the rational equivariant cohomology of that action is a free module with
	respect to its natural module structure over the cohomology ring of the classifying space of $G$. The question of whether or not the equivariant cohomology 
	of an action is free has already been considered by Borel \cite{Borel-seminar}, but it was probably only after the publication of the seminal work of Goresky, Kottwitz,
	and MacPherson that this notion gained considerable attention (indeed, the terminology equivariant formality first appeared in \cite{GoreskyKottwitzMacPherson}). This is partly
	explained by the fact that, in the same article, it was also shown how the entire equivariant cohomology ring of certain equivariantly formal actions can be recovered from 
	the knowledge of the $0$- and $1$-dimensional orbits. This technique of describing the equivariant cohomology ring has found widespread applications and 
	spawned many generalizations which are frequently subsumed under the name of GKM theory, see e.\,g.\,\cite{GuilleminZara-one-skeleton,GuilleminHolmTara,HaradaHenriquesHolm,GuilleminHolmZara,Mare-quaternionic,GuilleminSabatiniZara,MareWillems-octonionic,GoertschesMare-nonabelian-gkm,GoertschesWiemeler,Kuroki,GoertschesKonstantisZoller-non-kaehler,GoertschesKonstantisZoller-non-rigid,GoertschesKonstantisZoller-low-dimensional-gkm,EscherGoertschesSearle}.
	
	It should be noted that the property of being equivariantly formal imposes severe restrictions on the possible topology of the space being acted upon. For example,
	if a torus acts on a compact space with finitely many fixed points, then the vanishing of the odd degree cohomology of that space is a necessary and sufficient
	criterion for equivariant formality. In view of this, it is not too surprising that only a few classes of equivariantly formal actions have been identified.
	
	A rich source of examples which in a sense form the basic building blocks of every action are \emph{isotropy actions} on homogeneous spaces $G/K$, that is,
	the action of $K$ on $G/K$ induced by the group law in $G$, and it is natural to ask when this action is equivariantly formal. A prominent example 
	admitting an affirmative answer to this question is when $G$ and $K$ are of equal rank, see \cite{GuilleminHolmZara}. Equivariantly formal isotropy actions 
	where $K$ is of rank $1$ where classified in \cite{Carlson-torus}, and the case of subgroups $K$ of corank $1$ was considered in \cite{CarlsonHe-corank-one}. 
	Further examples of isotropy actions which are always equivariantly formal are those in which $G/K$ is a symmetric space \cite{Goertsches-symmetric}, or in which $K$ is the fixed 
	point set of an automorphism of $G$ \cite{Goertsches-automorphisms}. Generalizing the case of symmetric spaces in another direction, it was shown in \cite{AmannKollross} and independently
	\cite{thesis} that the isotropy action is equivariantly formal if $K$ is the fixed point set of two commuting involutions. With these results in mind,
	one might then wonder whether the isotropy action is equivariantly formal if $G/K$ is allowed to be a $\Gamma$-symmetric space \cite{Lutz}, that is, if $K$ is the common fixed point 
	set of a finitely generated Abelian subgroup $\Gamma \subseteq \Aut(G)$. The main purpose of this note is to provide further evidence for this conjecture:

	\begin{theorem}\label{thm-main}
		Let $G/K$ be $\ZZ_2\times\ZZ_k$ symmetric space, with a presentation in which $G$ and $K$ are compact, connected Lie groups.
		Then the isotropy action of $K$ on $G/K$ is equivariantly formal. In particular, the space $G/K$ is formal in the sense of Rational Homotopy Theory.
	\end{theorem}
	
	Although a careful inspection the proof of \cref{thm-main} also shows that $G/K$ is formal, we will exclusively focus on proving that pair $(G, K)$ is isotropy formal, by which, 
	following \cite{CarlsonFok}, we shall mean that the isotropy action of $K$ on $G/K$ is equivariantly formal. This is justified by \cite[theorem 1.4]{CarlsonFok}, according to which 
	the space $G/K$ is necessarily formal in the sense of Rational Homotopy Theory if $(G, K)$ is isotropy formal.
	
	We now briefly comment on the organization of this note and the proof of \cref{thm-main}. As is often the case when proving statements about Lie groups, 
	the general result can be obtained from the corresponding result for simple Lie groups. This is also true for \cref{thm-main}, as we shall show in 
	\cref{sec-reduction}. For the remainder of this section, we hence assume that $G$ is simple. Let then $\sigma_1$ be an 
	involution and $\sigma_2$ a finite-order automorphism on $G$ which commutes with $\sigma_1$. Choose, respectively, maximal tori $T_1$, $T_2$, and $S$ in 
	$K_1 = (G^{\sigma_1})_0$, $K_2 = (G^{\sigma_2})_0$, and $K = (G^{\sigma_1} \cap G^{\sigma_2})_0$, the subscripts indicating that we are considering the connected 
	components containing the identity element. It is a well-known fact that $Z_G(T_i)$ is a maximal torus in $G$, but $Z_G(S)$ will in general be contained in more then one 
	maximal torus of $G$. This is one of the reasons why it is possible to classify automorphisms on simple Lie algebras, but difficult to classify commuting pairs of automorphisms. 
	It is also one of the reasons that makes it difficult to show that $(G, K)$ is isotropy formal: indeed, if $Z_G(S)$ is a maximal torus on a simple Lie group $G$, then it follows 
	immediately from \cref{prop-formality-if-centralizer-is-torus} that $(G, K)$ is isotropy formal.
	
	One is thus lead to investigating the extent as to which $Z_G(S)$ fails to be a maximal torus. To this end, fix a maximal torus $T$ of $G$ containing $S$,
	which we may assume to be both invariant under $\sigma_1$ and $\sigma_2$. Let $\lieg = \liet\oplus\bigoplus_{\alpha} \lieg_{\alpha}$ be the corresponding root space 
	decomposition; here, the sum runs over all roots of $G$, and $\lieg$ and $\liet$ are the complexifications of the Lie algebras of $G$ and $T$, respectively. 
	Elementary Lie theoretic considerations then show that $Z_{\lieg}(\lies) = \liet\oplus\bigoplus_{\alpha} \lieg_{\alpha}$, where this time the sum ranges over the set 
	$\Omega$ of all roots $\alpha$ with $\alpha|_{\lies} = 0$. One of the key observations in \cite{thesis} which allowed us to show that $(G, K)$ is isotropy formal if
	also $\sigma_2$ is an involution was that $\sigma_2$ reflects each root in $\Omega$. In fact, if $\sigma_2$ is an inner involution, then $\sigma_2$ must 
	be the product $\prod_{\alpha} s_{\alpha}$ of simple roots reflections $s_{\alpha}$ as $\alpha$ ranges over the positive roots in $\Omega$.  Moreover,
	the roots in $\Omega$ are \emph{strongly orthogonal}, that is to say, if $\alpha$ and $\beta$ are contained in $\Omega$, then neither $\alpha + \beta$ nor $\alpha - \beta$ are 
	roots, and so the factors in the product $\prod_{\alpha} s_{\alpha}$ commute pairwise. 
	
	In \cref{thm-reflection} we show that this carries over to the case where $\sigma_2$ is an arbitrary automorphism of finite-order. That is, we show that $\sigma_2$ reflects all roots in 
	$\Omega$. While for pairs of involutions this is a fairly
	straightforward consequence of the decomposition of $\lieg$ into eigenspaces of $\sigma_1$ and $\sigma_2$, the proof for general $\sigma_2$ is much more involved and 
	occupies all of \cref{sec-reflection}. Generalizing \cite[proposition 3.2]{thesis} we also
	prove that if $\sigma_2$ reflects all roots in $\Omega$, then $\sigma_2$ too must be a product of simple root reflections $s_{\alpha}$ with $\alpha \in \Omega$, possibly 
	composed with an automorphism of the Dynkin diagram of $G$. We also observe that this Dynkin diagram automorphism commutes with each of the reflections $s_{\alpha}$ 
	and that the roots in $\Omega$ are again strongly orthogonal. 
	
	A striking consequence of these observations is that, on simple Lie groups with isomorphism type different from $\liealg{d}_4$, 
	$\sigma_2$ must necessarily be an involution on $T$ (or the identity map). This raises the following question: can one alter $\sigma_2$ in such a way that $\sigma_2$ 
	becomes an involution on all of $G$ while keeping $S$ as a maximal torus for the altered $K$? This is equivalent to asking whether the Weyl group element 
	$\prod_{\alpha} s_{\alpha}$, $\alpha$ a positive root in $\Omega$, can be lifted to an involutive inner automorphism of $G$. In general this is not possible, 
	but nonetheless, we observe in \cref{sec-case-same-type} that $\prod_{\alpha} s_{\alpha}$ does lift to an involution whenever $\sigma_1$ and $\sigma_2$ represent
	the same automorphism in the outer automorphism group $\Out(G)$ of $G$. This, combined with equivariant formality of isotropy actions of $\ZZ_2\times\ZZ_2$ 
	symmetric spaces, then enables us to deduce that the pair $(G, K)$ is isotropy formal if either $\sigma_1$ and $\sigma_2$ are both inner or both outer automorphisms (\cref{cor-eq-formality-in-good-cases}), or if 
	$\sigma_1$ is an outer automorphism and $\sigma_2$ is inner (\cref{thm-formality-if-involution-is-outer-other-is-inner}).
	
	The remaining case to discuss is when $\sigma_1$ is an inner automorphism, but $\sigma_2$ is not. Lacking a general argument, we explicitly
	determine the set $\Omega$ in this case, see \cref{sec-involution-inner-other-outer}. The overall goal here is to use the explicit description of $\sigma_2$ on $T$ to recover the maximal torus $S$ of $K$, as showing 
	that $(G, K)$ is isotropy formal is equivalent to showing that so is $(G, S)$. Fortunately, there are only three series of simple Lie groups on which $\sigma_2$
	can be an outer automorphism, and the possibilities for the set $\Omega$ are limited as well. This can be seen as follows. If one fixes a root $\alpha \in \Omega$, then there is a 
	notion of positivity  on the set of all roots for which $\alpha$ becomes dominant, and then the remaining roots in $\Omega - \{\pm \alpha\}$ must be contained in the 
	subdiagram of the Dynkin diagram of $G$ that is spanned by all roots which are perpendicular to $\alpha$. Except for root systems of type $\liealg{d}_4$, this subdiagram will 
	contain at most two connected components, and for the classical Lie groups one of these components will belong to the same series as $G$. Proceeding recursively one finds 
	that there is a chain of nested subdiagrams of the Dynkin diagram of $G$ in which the roots of $\Omega$ are contained, and this significantly limits the sets $\Omega$ that 
	might occur, at least up to application of a Weyl group element (this was already observed in \cite{Kostant-cascade} where such sets are called chain cascades). 	
	Further obstructions resulting from the involutivity of $\sigma_1$ provide us with a short list of cases for which we 
	verify that $(G, S)$ must be isotropy formal, and this
	finishes the proof of \cref{thm-main}.
	
	\subsection*{Acknowledgements} 
	This work was supported by the SNSF-Project 200020E\_193062 and the DFG-Priority programme SPP 2026. The author would also like to thank Philipp Reiser for sharing his
	proof of the description of subgroups of $\ZZ_2\times\ZZ_k$. 
	
	\section{Preliminaries}\label{sec-preliminaries}
	
	This section summarizes some known results about automorphisms on simple Lie algebras and equivariant cohomology of isotropy actions which we will use later on. 
	A concise overview of many aspects of equivariant cohomology can be found in \cite{Quillen-spectrum-equivariant-cohomology} or the survey \cite{GoertschesZoller}, and 
	a more comprehensive exposition is provided by the book \cite{AlldayPuppe}. Our main reference for results concerning automorphisms on Lie algebras are \cite[Chapter X]{Helgason} 
	and \cite[section 5]{GrayWolf-homogeneous-spaces-I}.

	\subsection{Automorphisms on Lie algebras}
	We will frequently have to make use of the classification of finite-order automorphisms and their fixed points sets on 
	(finite-dimensional) complex simple Lie algebras \cite[section X.5]{Helgason}, \cite[p.\,106]{GrayWolf-homogeneous-spaces-I}. Given such a 
	Lie algebra $\liealg{u}$ and a Cartan subalgebra $\lieh \subseteq \liealg{u}$, let $\Delta \subseteq \lieh^{\ast}$ be the roots and $\Pi \subseteq \Delta$ be the simple roots with 
	respect to some notion of positivity. There exist elements $X_{\alpha} \in \lieg_{\alpha}$, $Y_{\alpha} \in \lieg_{-\alpha}$ for each $\alpha \in \Pi$, called 
	canonical generators in \cite[proposition X.4.1]{Helgason} and Weyl basis in \cite[lemma 5.3]{GrayWolf-homogeneous-spaces-I}, which generate $\liealg{u}$ as an algebra. If 
	$\nu\colon\Pi\rightarrow\Pi$ is an automorphism of the Dynkin diagram of $\liealg{u}$, then an automorphism $\nu$ of $\liealg{u}$ is determined by requiring that 
	$\nu(X_{\alpha}) = X_{\nu(\alpha)}$ and $\nu(Y_{\alpha}) = Y_{\nu(\alpha)}$ holds for all simple roots $\alpha$. We say that $\nu\colon\liealg{u}\rightarrow\liealg{u}$ is an 
	automorphism of $\liealg{u}$ \emph{induced by an automorphism of the Dynkin diagram} of $\liealg{u}$. Up to conjugation, every automorphism $\mu$ on $\liealg{u}$ then
	is of the form $\mu = c\circ\nu$ for some automorphism $\nu$ induced by an automorphism of the Dynkin diagram, where $c$ is an inner automorphism which commutes with
	$\nu$ and fixes $\lieh$ pointwise, cf.\,\cite[lemma 5.3]{GrayWolf-homogeneous-spaces-I}. In particular, the rank of $\liealg{u}^{\mu}$ must be equal to the rank of
	$\liealg{u}^{\nu}$ for some automorphism $\nu$ which is induced by an automorphism of the Dynkin diagram of $\liealg{u}$. Combined with \cite[table I, p.\,505]{Helgason}
	this proves the following well-known
	
	\begin{proposition}\label{prop-rank-of-fixed-point-set-of-automorphism}
		Let $\liealg{u}$ be a simple complex Lie algebra and $\mu\colon\liealg{u}\rightarrow\liealg{u}$ an automorphism of Lie algebras induced by a $k$-th order automorphism of
		the Dynkin diagram of $\liealg{u}$, $k > 1$. Then the isomorphism types of $\liealg{u}$ and of the fixed point set $\liealg{u}_0$ of $\mu$ are contained in the 
		following table:
		\begin{center}
			\begin{tabular}{l|c|c|c|c|c}
				$\liealg{u}$				& $\liealg{a}_{2n}$ ($k = 2$)	& $\liealg{a}_{2n-1}$ ($k=2$)	& $\liealg{d}_n$ ($n \geq 4, k= 2$)	& $\liealg{d}_4$ ($k = 3$)		& $\liealg{e}_6$ ($k=2$)	\\
				\hline
				$\liealg{u}_0$	& $\liealg{b}_n$		& $\liealg{c}_n$			& $\liealg{b}_{n-1}$			& $\liealg{g}_2$		& $\liealg{f}_4$
			\end{tabular}
		\end{center}
		Moreover, the rank of the fixed point set of any non-inner automorphism on $\liealg{u}$ must be equal to the rank of one of these Lie 
		algebras $\liealg{u}_0$.
	\end{proposition}
	
	For finite-order automorphisms on $\liealg{u}$, we can say even more. Choose a Dynkin diagram automorphism of order $k$ on $\liealg{u}$ and
	let $\nu\colon\liealg{u}\rightarrow\liealg{u}$ be the induced automorphism of Lie algebras. Let $D$ be the diagram in Table $k$ on \cite[p.\,503]{Helgason} corresponding to the 
	isomorphism type of $\liealg{u}$ and let further $\alpha_1, \ldots, \alpha_n$ be simple roots for $\liealg{u}^{\nu}$. There exists a vertex $v \in V(D)$ such that
	the diagram $D - \{v\}$ is connected and equal to the Dynkin diagram of $\liealg{u}^{\nu}$. Let then $a_1, \ldots, a_n$ be the labels of the vertices corresponding to 
	$\alpha_1, \ldots, \alpha_n$ and $a_0$ the remaining label. Every choice of non-negative 
	integers $s_0, \ldots, s_n$ with no nontrivial common factor then defines an automorphism of order $k(a_0s_0 + \ldots + a_ns_n)$ on $\liealg{u}$. 
	If $J \subseteq \{0, \ldots, n\}$ is the set of indices $j$ such $s_j = 0$, then the fixed point set of this automorphism has an $(n - |J|)$-dimensional center, and the semisimple 
	part is the subdiagram of $D$ with vertex set $J$. Up to conjugation, every finite-order automorphism on $\liealg{u}$ arises in this way \cite[theorem X.5.15]{Helgason}.
	If one carries out this procedure for automorphisms of order $2$ one obtains Tables II and III on \cite[pp.\,515]{Helgason}. For the convenience of the reader, we reproduce these tables in
	
	\begin{proposition}\label{prop-classification-of-involutions}
		Let $\mu$ be an involution on the simple complex Lie algebra $\liealg{u}$ and $\liealg{u}_0$ its fixed point set. Then one of the following cases occurs.
		\begin{enumerate}
			\item 
			$\mu$ is an inner automorphism and $\liealg{u}_0$ is semisimple. The isomorphism types of $\liealg{u}$ and $\liealg{u}_0$ are contained in the following table.
			\par\bigskip\begin{raggedleft}
				\begin{tabular}{l|c|c|c|c|c|c|c|c}
					$\liealg{u}$	& $\liealg{b}_n$ 	& $\liealg{c}_n$ 	& $\liealg{d}_n$		& $\liealg{g}_2$		& $\liealg{f}_4$
						& $\liealg{e}_6$		& $\liealg{e}_7$		& $\liealg{e}_8$		\\
						\hline
					$\liealg{u}_0$	& $\liealg{d}_p\oplus\liealg{b}_{n-p}$	& $\liealg{c}_p\oplus\liealg{c}_{n-p}$	 & $\liealg{d}_p\oplus\liealg{d}_{n-p}$
						& $\liealg{a}_1\oplus\liealg{a}_1$	& $\liealg{b}_4$	  & $\liealg{a}_1\oplus\liealg{a}_5$
						& $\liealg{a}_7$	 & $\liealg{a}_1\oplus\liealg{e}_7$	\\
									& & & & & $\liealg{a}_1\oplus\liealg{c}_3$ & & $\liealg{a}_1\oplus\liealg{d}_6$ & $\liealg{d}_8$
				\end{tabular}
			\end{raggedleft}
			\par\bigskip\noindent Here we have $2 \leq p \leq n$ (case $\liealg{b}_n$), $1 \leq p \leq \lfloor \frac{n}{2} \rfloor$ (case $\liealg{c}_n$),
			and $2 \leq p \leq \lfloor \frac{n}{2} \rfloor$ (case $\liealg{d}_n$).
			
			\item
			$\mu$ is an inner automorphism and $\liealg{u}_0$ has a one-dimensional center. The isomorphism types of $\liealg{u}$ and the semisimple part
			$[\liealg{u}_0, \liealg{u}_0]$ are given as follows.
			\par\bigskip\begin{raggedleft}
				\begin{tabular}{l|c|c|c|c|c|c|c}
					$\liealg{u}$	& $\liealg{a}_n$		& $\liealg{b}_n$ ($n > 2$)	& $\liealg{c}_n$ ($n > 1$)	& $\liealg{d}_4$	& $\liealg{d}_n$ ($n > 4$)	& $\liealg{e}_6$		& $\liealg{e}_7$		\\
					\hline
					$[\liealg{u}_0, \liealg{u}_0]$		& $\liealg{a}_p\oplus\liealg{a}_{n-p-1}$	& $\liealg{b}_{n-1}$		& $\liealg{a}_{n-1}$
						& $\liealg{a}_3$		& $\liealg{d}_{n-1}$		& $\liealg{d}_5$		& $\liealg{e}_6$ \\
													& & & & & $\liealg{a}_{n-1}$
				\end{tabular}	
			\end{raggedleft}
			\par\bigskip\noindent The possible value of $p$ are $0 \leq p \leq \lfloor\frac{n-1}{2}\rfloor$.
			
			\item 
			$\mu$ is not an inner automorphism and $\liealg{u}_0$ is semisimple. The types of $\liealg{u}$ and $\liealg{u}_0$ are 
			\par\bigskip\begin{raggedleft}
				\begin{tabular}{l|c|c|c|c}
					$\liealg{u}$		&	$\liealg{a}_{2n}$ ($n \geq 2$)	& $\liealg{a}_{2n-1}$ ($n > 2$)	& $\liealg{d}_{n+1}$ ($n > 1$)	& $\liealg{e}_6$ 	\\
					\hline
					$\liealg{u}_0$	& $\liealg{b}_n$		& $\liealg{d}_n$		& $\liealg{b}_p\oplus\liealg{b}_{n-p}$	& $\liealg{c}_4$		\\
									&					& $\liealg{c}_n$		&										& $\liealg{f}_4$
				\end{tabular}
			\end{raggedleft}
			\par\bigskip\noindent and $0 \leq p \leq \lfloor \frac{n}{2} \rfloor$.
		\end{enumerate}
	\end{proposition}

	\subsection{Equivariant cohomology}
	Let $G$ be a Lie group and fix a contractible space $EG$ on which $G$ acts freely. To any $G$-space $(X, \rho)$, that is, to any toplogical space $X$ endowed with a fixed 
	continuous $G$-action $\rho$, we may associate the \emph{(Borel) equivariant cohomology}, which by definition is the graded group
	\begin{align*}
		H^{\ast}_{G}(X; R) := H^{\ast}(EG\times_{(G,\rho)}X; R).
	\end{align*}
	In other words, the Borel equivariant cohomology is the (singular) cohomology of the orbit space $EG\times_{(G,\rho)} X$ of the diagonal action of $G$ on $EG\times X$.
	
	Let now $R = \RR$ be the real numbers. This is the coefficient group that we will be mostly interested in and which we will thus from now on supress in the notation. Note that 
	$H_G^{\ast}(X) = H_G^{\ast}(X; \RR)$ is a ring via the cup product and that the assignment $X \mapsto H_G^{\ast}(X)$ is functorial. In particular, any map $X \rightarrow \{\ast\}$ 
	induces a morphism of rings $H_G^{\ast}(\ast) \rightarrow H_G^{\ast}(X)$ via which $H_G^{\ast}(X)$ becomes a $H_G^{\ast}(\ast)$-module. If $H_G^{\ast}(X)$ happens to be a 
	free $H_G^{\ast}(\ast)$-module, then we say that the $G$-action on $X$ is \emph{equivariantly formal}.

	There are various other equivalent ways to describe equivariantly formal actions, cf.\,for example \cite[section 7]{GoertschesZoller} or
	\cite{CarlsonFok} for the special case of isotropy actions (note that in \cite{CarlsonFok} 
	rational coefficients are used, but that this does not affect the question of equivariant formality, see \cite[corollary 2.5, proposition 2.7]{GoertschesMare-hyperpolar}). We will 
	mostly use two criteria which are valid in the context of homogeneous spaces, the first one being 
	a fairly straightforward consequence of \cite[p.\,46]{Hsiang}:
	
	\begin{proposition}\label{pop-formality-betti-numbers}
		Let $G$ be a compact, connected Lie group and $K$ a closed, connected subgroup. For any maximal torus $S \subseteq K$ we have
		\begin{align*}
			2^{\rank G - \rank K}\cdot \dim H^0\left(\tfrac{N_G(S)}{N_K(S)}\right) &\leq \dim H^{\ast}(G/K)
		\end{align*}
		with equality if and only if $(G, K)$ is isotropy formal.
	\end{proposition} 

	The second criterion is very useful for passing between different subgroups of $G$. It allows us to deduce
	equivariant formality of a pair $(G, K)$ from that of a pair $(G, H)$, provided that $K$ and $H$ share a common maximal torus.
	
	\begin{proposition}[{\cite[theorem 1.1]{Carlson-torus}}]\label{prop-formality-depends-on-torus}
		Let $G$ be a compact, connected Lie group, $K$ a closed connected subgroup, and $S \subseteq K$ maximal torus. Then
		$(G, K)$ is isotropy formal if and only if $(G, S)$ is isotropy formal.
	\end{proposition}

	Recall that a fibration is totally nonhomologous to zero if the inclusion of the fiber into the total space induces a surjection on cohomology.
	For a closed connected subgroup $K$ of a compact connected Lie group $G$ and the fibration $K \hookrightarrow G \rightarrow G/K$ one
	can equivalently demand that restriction induces a surjection from the space $P_{\lieg}$ of invariant polynomials on $\lieg$ onto
	the space $P_{\liek}$ of invariant polynomials on $\liek$ \cite[theorem VI, section 11.6]{GreubHalperinVanstone}, and in this case one also
	says that $K$ is totally nonhomologous to zero in $G$. A classical result states that $G/K$ then is formal in the sense of Rational Homotopy Theory
	\cite[corollary I, section 10.19]{GreubHalperinVanstone}, and it turns out that this condition is also sufficient for $(G, K)$ to be isotropy formal.
	
	\begin{proposition}[{\cite[corollary 4.2]{Shiga}}]\label{prop-tnhz-implies-isotropy-formality}
		Let $K$ be a compact, connected subgroup of a compact, connected Lie group $G$. If $K$ is totally nonhomologous to zero in $G$, then the pair $(G, K)$ is isotropy formal.
	\end{proposition}
	
	Given an automorphism $\sigma\colon G\rightarrow G$ one a simple Lie group $G$ one can always find an automorphism $\tau$ which is induced by an automorphism of
	the Dynkin diagram and such that $K = G^{\sigma}$ and $H = G^{\tau}$ share a maximal torus. This implies that $(G, K)$ is isotropy formal, because of 
	\cref{prop-formality-depends-on-torus} and the following 
	
	\begin{proposition}[{\cite[section 4]{Goertsches-automorphisms}}]\label{prop-dynkin-diagram-auto-tnhz}
		Let $G$ be a compact, connected, simple Lie group and $\mu\colon G\rightarrow G$ an automorphism induced by an automorphism of the Dynkin
		diagram of $G$. Then $K = (G^{\mu})_0$ is totally nonhomologous to zero in $G$.
	\end{proposition}
	
	Finally, to put \cref{thm-main} into perspective, we mention the following immediate consequence of equivariant formality of isotropy actions of 
	$\ZZ_{\ell}$--symmetric spaces proved in \cite{Goertsches-automorphisms}. 
	
	\begin{proposition}
		Let $G$ be a compact, connected Lie group and $\Gamma = \ZZ_p\times \ZZ_q$, $\Gamma \subseteq \Aut(G)$,
		a subgroup with $(p, q) = 1$. Then the pair $(G, K)$ with $K = (G^{\Gamma})_0$ is isotropy formal.
	\end{proposition}
	\begin{proof}
		This follows readily from \cite[theorem 1.1]{Goertsches-automorphisms}, because $\Gamma \cong \ZZ_{pq}$.
	\end{proof}

	\section{From arbitrary quotients to quotients of simple Lie groups}\label{sec-reduction}
	
	Our first goal will be to show that in order to prove \cref{thm-main} it suffices to consider homogeneous spaces in which the the nominator
	is a simple Lie group. Thus, assume that \cref{thm-main} is true for all $\ZZ_2\times\ZZ_k$-symmetric
	spaces $G/K$ in which $G$ is a compact connected simple Lie group. We claim that under this assumption also for every subgroup $\Gamma \subseteq \ZZ_2\times\ZZ_k$
	and for every compact connected simple Lie group $L$ on which $\Gamma$ acts via automorphisms the pair $(L, U)$, with $U$ the identity component of the fixed point set of 
	$\Gamma$, is isotropy formal. This will be a consequence of equivariant formality of isotropy actions of $\ZZ_{\ell}$-symmetric spaces 
	(\cite[theorem 1.1]{Goertsches-automorphisms}) once we prove that $\Gamma$ is either of the form $\ZZ_{\ell}$ or $\ZZ_2\times \ZZ_r$ for some $r$. 
	
	To see that $\Gamma$ has the claimed form, note that $\Gamma$ is generated by at most two elements, because every subgroup of $\ZZ\times \ZZ$ is
	generated by two elements by \cite[theorem I.7.3]{Lang}. If $\Gamma$ is contained in $\{0\}\times \ZZ_k$, then $\Gamma$ is cyclic, because every subgroup of $\ZZ_k$
	is cyclic, and this implies that we may assume
	that $\Gamma$ is generated by two elements of the form $x = (\overline{1}, \overline{a})$, $y = (0, \overline{b})$ with $a,b \in \ZZ$. Let then $g$ be a greatest common divisor 
	of $a$ and $b$. Since $\ZZ$ is a principal ideal domain, $g = ma + nb$ for integers $m$ and $n$ \cite[proposition II.5.1]{Lang}. In particular, $mx + ny = (\overline{m}, \overline{g})$ is contained in 
	$\Gamma$. Suppose first that $m$ is even. Then $a/g\cdot (\overline{m}, \overline{g}) = (0, \overline{a})$ too is an element of $\Gamma$. This means that $(\overline{1}, 0)$ is contained in 
	$\Gamma$, and then we must have $\Gamma = \ZZ_2\times\Gamma_2$ for $\Gamma_2 = \{ (0, \overline{z}) \in \Gamma \,|\, z \in \ZZ\}$. As $\Gamma_2$ is cyclic, the claim 
	follows in this case. Next, assume that $\overline{m} = \overline{1}$. If $a/g$ is even, then $a/g\cdot (\overline{1}, \overline{g}) = (0, \overline{a})$ as before. Similarly, if $b/g$ is odd, then 
	$b/g\cdot (\overline{1}, \overline{g}) = (\overline{1}, \overline{b})$ shows that
	$(\overline{1}, 0)$ is contained in $\Gamma$. Thus, we are left with considering the case that $a/g$ is odd while $b/g$ is even. Then $a/g\cdot (\overline{1}, \overline{g}) = (\overline{1}, \overline{a})$
	and $b/g\cdot (\overline{1}, \overline{g}) = (0, \overline{b})$, whence $\Gamma$ is cyclic with generator $(\overline{1}, \overline{g})$. In any case, $\Gamma$ is as claimed.	
	
	Therefore, if we assume that \cref{thm-main} is true for all homogeneous spaces of the form $G/K$ with $G$ simple, then the following 
	general reduction principle shows that \cref{thm-main} holds true for all $\ZZ_2\times\ZZ_k$-symmetric spaces $G/K$.
	
	\begin{theorem}\label{thm-reduction-principle}
		Let $\Gamma$ be a finite group having the following property: for every compact connected simple Lie group $L$ on which $\Gamma$ acts via automorphisms
		and for every subgroup $\Gamma' \subseteq \Gamma$ the pair $(L, U)$, with $U$ the identity component of $L^{\Gamma'}$, is isotropy formal. Then for 
		every compact connected Lie group $G$ and for every action of $\Gamma$ on $G$ by automorphisms the pair $(G, K)$ with $K = (G^{\Gamma})_0$ is isotropy formal.
	\end{theorem}
	
	Similar reduction principles already appeared in \cite{Goertsches-automorphisms} and \cite{thesis}. However, there an ad-hoc argument was used to deduce that 
	if a compact connected Lie group pair $(G, K)$ is isotropy formal, also the pair $(G\times\ldots\times G, \Delta(K))$ is isotropy formal. This is 
	generally true:
	
	\begin{proposition}\label{prop-diagonal-actions}
		Let $G$ be a compact, connected, semisimple Lie group and $K$ a closed, connected subgroup. If the isotropy
		action of $K$ on $G/K$ is equivariantly formal, then the isotropy action of $\Delta(K)$
		on $(G\times\ldots\times G)/\Delta(K)$ is equivariantly formal as well.
	\end{proposition}
	\begin{proof}
		For notational convenience, we only consider the case of two factors. By \cref{prop-formality-depends-on-torus} there will be no loss of generality if we assume that $K$ is a 
		toral subgroup, and because formality of the pair $(G, K)$ only depends on the inclusion of Lie algebras $\liek \hookrightarrow \lieg$ 
		\cite[proposition 2.13]{GoertschesMare-hyperpolar}, there is also no loss of generality if we assume $G$ to be simply-connected. Then according to 
		\cref{pop-formality-betti-numbers} we have
		\begin{align*}
			2^{\rank G - \dim K}\cdot \dim H^0(N_G(K)) &\leq \dim H^{\ast}(G/K)
		\end{align*}
		with equality if and only if $(G, K)$ is isotropy formal. Similar statements hold with $G\times G$ and $\Delta(K)$ in place of $G$ and $K$, and thus, we need to 
		compare $N_G(K)$ and $N_{G\times G}(\Delta(K))$. So suppose that $(g, h) \in G\times G$ normalizes $\Delta(K)$. This is equivalent to requiring that $g$ and $h$ 
		normalize $K$, and that $c_g(k) = c_h(k)$ holds for all $k \in K$; here $c_g$ and $c_h$ denote the conjugation maps in $G$. In particular, $g^{-1}h$ centralizes $K$, and so 
		is contained in $Z_G(K)$. It follows that 
		\begin{align*}
			N_{G\times G}(\Delta(K)) &= \left\{ (g, gz) \,|\, g \in N_G(K)\text{, }z \in Z_G(K)\right\}.
		\end{align*}
		However, since $K$ is a torus, the centralizer $Z_G(K)$ of $K$ in $G$ is connected \cite[corollary IV.4.51]{Knapp}, and therefore 
		$N_{G\times G}(\Delta(K))$ and $N_G(K)$ have the same number of connected components. Now consider 
		the fibration $(G\times G)/\Delta(K) \rightarrow (G\times G)/\Delta(G)$ with fiber $G/K$. Since 
		$G$ is simply-connected, the second page of the associated spectral sequence $(E_r^{\ast,\ast})_{r\geq 2}$ reads
		\begin{align*}
			E_2^{p,q} = H^q(G/K)\otimes H^p\bigl(\tfrac{G\times G}{\Delta(G)}\bigr),
		\end{align*}
		and if $K$ acts equivariantly formally on $G/K$, then because $(G\times G)/\Delta(G)$ is diffeomorphic to $G$ and 
		because the above spectral sequence abuts to $(G\times G)/\Delta(K)$, we have 
		\begin{align*}
			\dim H^{\ast}\bigl(\tfrac{G\times G}{\Delta(K)}\bigr) &\leq \dim H^{\ast}(G/K)\otimes H^{\ast}\bigl(\tfrac{G\times G}{\Delta(G)}\bigr) \\
			&= 2^{\rank G}\cdot \dim H^{\ast}((G/K)^K) \\
			&= 2^{\rank (G\times G) - \dim K}\cdot \dim H^0(N_{G\times G}(\Delta(K))).
		\end{align*}
		As remarked earlier, this means that $\Delta(K)$ acts equivariantly formally on $(G\times G)/\Delta(K)$.
	\end{proof}	
	
	\begin{proof}[Proof of \cref{thm-reduction-principle}]
		Let $G$ be a compact connected Lie group on which $\Gamma$ acts via automorphisms. As already observed in the proof of \cref{prop-diagonal-actions},
		by \cite[proposition 2.13]{GoertschesMare-hyperpolar} we may pass to covers of $G$, and hence there is no loss of generality if we assume that $G$ splits as a 
		direct product $G = Z\cdot\prod_{j=1}^r G_j$ with simply-connected simple factors $G_i$ and toral factor $Z$. Since $\Gamma$ acts via
		automorphisms on $G$, $\Gamma$ will permute the simple factors $\mathcal{G} = \{G_1, \ldots, G_r\}$ of $G$ and preserve $Z$. Moreover,
		it follows at once from \cref{pop-formality-betti-numbers} that for any subtorus $Z' \subseteq Z$ the pair $(Z, Z')$ is isotropy formal, and this implies that
		we may assume that $Z$ is trivial, that is, that $G$ is semisimple. 
		
		Thinking of $\Gamma$ temporarily as a subgroup of the permutation group of $\mathcal{G}$, we fix a simple factor $L \in \mathcal{G}$ and enumerate the 
		elements of the orbit $\Gamma\cdot L$ through $L$ as $\Gamma\cdot L = \{L, \gamma_1(L), \ldots, \gamma_t(L)\}$ for certain automorphisms 
		$\gamma_1, \ldots, \gamma_t \in \Gamma$. The subgroup $U = L\cdot \gamma_1(L)\cdots \gamma_t(L)$ is a $\Gamma$-invariant subspace of $G$, and the fixed point set 
		$U \cap K$ of $\Gamma$ is a subgroup of $K$. By considering the various orbits of $\Gamma$ on $\mathcal{G}$ we see that $G/K$ splits as a 
		product of homogeneous spaces of the form $U/(U\cap K)$ with $U$ as above and that $(G, K)$ is isotropy formal if and only if so are all the pairs 
		$(U, U\cap K)$ arising in this way. Hence, we may assume from the outset that $G$ is a direct product $G = L\cdot \gamma_1(L)\cdots \gamma_t(L)$ and that 
		$\mathcal{G} = \{L, \gamma_1(L), \ldots, \gamma_t(L)\}$ are the simple factors of $G$.
		
		Now set $\gamma_0 := \id_L$ and observe the we have an isomorphism of Lie groups
		\begin{align*}
			\Phi\colon L\times \ldots\times L &\rightarrow G, \\
			(g_0, \ldots, g_t) &\mapsto \prod_{j=0}^t \gamma_j.g_j.
		\end{align*}	
		To identifty the isomorphic image of the fixed point set $K$ under this map, let $\Gamma' \subseteq \Gamma$ be the set of all elements which map $L$ into itself. In other
		words, $\Gamma'$ is the isotropy subgroup at the point $L \in \mathcal{G}$ and the orbit map $\Gamma/\Gamma' \rightarrow \mathcal{G}$, 
		$\gamma\Gamma' \mapsto \gamma (L)$, becomes a bijection. Denote by $H \subseteq L$ the fixed point set of $\Gamma'$ on $L$. Then $\Delta(H) \cong K$ under the 
		isomorphism above. Indeed, for each $\gamma \in \Gamma$ and all $j$ there exists one and exactly one $\pi(j)$ such that 
		$\gamma\gamma_j (L) = \gamma_{\pi(j)} (L)$. By choice of $\Gamma'$, this means that $\gamma\gamma_j = \gamma_{\pi(j)}\gamma'$ for some element 
		$\gamma' \in \Gamma'$, and so, if $h \in H$ is given, then  
		\begin{align*}
			\gamma.(\gamma_j.h) &= (\gamma_{\pi(j)}\gamma').h = \gamma_{\pi(j)}.h.
		\end{align*}
		Using once more that $\Gamma$ acts via automorphisms on $G$, it follows that 
		\begin{align*}
			\gamma.\Phi(h, \ldots, h) &= \prod_{j=0}^t \gamma.(\gamma_j.h) = \prod_{i=0}^t \gamma_i.h = \Phi(h, \ldots, h).
		\end{align*}
		Conversly, suppose that $\Phi(g_0, \ldots, g_t)$ is fixed by $\Gamma$. In particular, each element $\gamma' \in \Gamma'$  must fix 
		$\Phi(g_0, \ldots, g_t)$. But $\gamma'$ restricts to an automorphism of the simple factor 
		$L$, which implies that $g_0 = \gamma'.g_0$. Similarly, the element 
		$(\gamma_i)^{-1}$ sends the simple factor $\gamma_i L$ onto the simple factor $L$, and this implies that $g_0 = g_i$. Putting
		these two observations together shows that $\Delta(H) \cong K$. Since by assumption $(L, H)$ is isotropy formal, also
		$(G, K)$ is isotropy formal by \cref{prop-diagonal-actions}.
	\end{proof}			
	
	\section{Normal forms}\label{sec-normal-form}
	
	Let $G$ be a compact connected Lie group and $\sigma_1$, $\sigma_2$ two commuting (Lie group) automorphisms of order $2$ and $k$, respectively. These
	two automorphisms define a $\ZZ_2\times\ZZ_k$-action on $G$ by automorphisms, and conversely, any such action arises in this way. 
	Let then $K$ be the identity component of $G^{\sigma_1} \cap G^{\sigma_2}$ and $S \subseteq K$ a maximal torus. By \cref{prop-formality-depends-on-torus}
	the pair $(G, K)$ is isotropy formal precisely when $(G, S)$ is isotropy formal, so one strategy to prove \cref{thm-main} is to obtain a description
	of $S$ in terms of computable data on $G$ and to use this description to prove that $(G, S)$ is isotropy formal. When $k = 2$, i.\,e.\,when $\sigma_2$ is an involution and 
	$G/K$ is a $\ZZ_2\times\ZZ_2$-symmetric space, this is what was done in \cite[Chapter II]{thesis}. It was observed there that it is possible to find a suitable $\sigma_2$-invariant
	maximal torus $T$ of $G$ which contains $S$ and a notion of positivity on the set of roots of $G$ in such a way that, on the complexification $\liet$ of the Lie algebra $\liet_{\RR}$ 
	of $T$, the automorphism $\sigma_2$ is the composition of $\prod_{\alpha\in\Omega^{+}} s_{\alpha}$ 
	with an automorphism of $G$ which permutes the positive roots and which commutes with each root reflection $s_{\alpha}\colon\liet\rightarrow\liet$, that is, the 
	reflection along the hyperplane $\ker \alpha$, $\alpha \in \Omega^{+}$. The set $\Omega^{+}$ consists of all positive roots which restrict to zero on the complexification $\liealg{s}$ of the Lie algebra of $S$, and this set 
	$\Omega^{+}$ can be determined explicitly in terms of the root system of $G$. In this way a maximal torus $S$ of $K$ can be recovered from the root system
	of $G$.
	
	The purpose of the next two sections is to show that similar results hold for arbitrary $\ZZ_2\times\ZZ_k$-symmetric spaces. Starting from a maximal torus 
	$S$ of $K$ we will construct a maximal torus $T$ of $G$ which is invariant under $\sigma_1$ and $\sigma_2$. We say that the pair $\sigma_1$, $\sigma_2$ 
	has property $\star$ if with respect to this maximal torus $T$ every root that vanishes on $S$ is reflected by $\sigma_2$ (\cref{def-property-reflective}),
	and we will show in \cref{thm-normal-form-automorphisms} below that in this case $\sigma_2|_T$ necessarily is a product of root reflections, possibly composed with an 
	automorphism of the Dynkin diagram of $G$. In the next section (cf. \cref{thm-reflection}) we will then show that on simple Lie groups $G$, every pair of automorphisms 
	$\sigma_1$, $\sigma_2$ defining a $\ZZ_2\times\ZZ_k$-action on $G$ actually satisfies property $\star$.
	
	In what follows, we will often have to refer to roots and root systems. These concepts are most naturally defined for complex Lie algebras, so for notational convenience we
	reserve undecorated fraktur letters to denote the complexification of the Lie algebra of the corresponding Lie group. If we want to refer to the underlying
	real Lie algebra, we add a subscript. For example, if $G$ is a Lie group, then $\lieg_{\RR}$ is the Lie algebra of $G$ and $\lieg = (\lieg_{\RR})^{\CC}$ is its complexification.  	
	
	\begin{proposition}\label{prop-centralizers-of-subtori}
		Let $T$ be a maximal torus of a compact, connected Lie group $G$ and $S \subseteq T$ a subtorus. Let $\Delta \subseteq \liet^{\ast}$ be the roots of $\lieg$ with
		respect to $\liet$ and $\Omega \subseteq \Delta$ the set of roots vanishing on $\lies$. Then, if we denote for each root $\alpha$ its root space by $\lieg_{\alpha}$, the 
		centralizer $Z_{\lieg}(\lies)$ of $\lies$ in $\lieg$ is
		\begin{align*}
			Z_{\lieg}(\lies) = \liet\oplus\bigoplus_{\alpha\in\Omega}\lieg_{\alpha}.
		\end{align*}
	\end{proposition}
	
	Now let $G$ be a compact, connected Lie group and $\sigma_1$, $\sigma_2$ two automorphisms of $G$. We suppose that $\sigma_1$ is an involution, that
	$\sigma_2$ is of finite-order, and that $\sigma_1$ and $\sigma_2$ commute. Denote by $K$ the connected component of $G^{\sigma_1} \cap G^{\sigma_2}$
	containing the identity element and choose a maximal torus $S \subseteq K$. We also fix an $\Ad_G$-invariant negative-definite inner product $\langle\cdot,\cdot\rangle$ on 
	$\lieg_{\RR}$ which is invariant under $\sigma_1$ and $\sigma_2$, and we extend $\langle\cdot,\cdot\rangle$ $\CC$-linearly to $\lieg$. Such a product always exists, 
	as we may average over the finite subgroup of $\Aut(G)$ generated by $\sigma_1$ and $\sigma_2$. Let further $K_i = (G^{\sigma_i})_0$ be the identity component 
	of the fixed point set of $\sigma_i$. Note that because $\sigma_1$ and $\sigma_2$ commute, 
	$\sigma_2$ restricts to an automorphism $K_1 \rightarrow K_1$. As is well-known, $T_1 = Z_{K_1}(S)$ thus is a maximal torus of $K_1$ (see 
	e.g. \cite[lemma X.5.3]{Helgason}), and by the same reasoning $T= Z_G(T_1)$ is a maximal torus in $G$. Both $T_1$ and $T$ are invariant under
	$\sigma_1$ as well as $\sigma_2$. We write $\Delta$ for the $\liet$-roots of $\lieg$, $\Omega$ for the set of roots vanishing on $\lies$.
	
	\begin{proposition}\label{prop-formality-if-centralizer-is-torus}
		Given a compact simple Lie group $G$  and two commuting automorphisms $\sigma_1$ and $\sigma_2$ on $G$, let $K$ be the identity component of
		$G^{\sigma_1} \cap G^{\sigma_2}$. Suppose that $Z_G(S)$ is a maximal torus in $G$. Then $(G, K)$ is isotropy formal.
	\end{proposition}
	\begin{proof}
		In this case $T = Z_G(S)$ by \cref{prop-centralizers-of-subtori}. Also note that $S$ contains regular elements for $T$. Pick one such element $X \in \lies$ and define a notion of positivity on the set $\Delta$ of $\liet$-roots by
		declaring a root $\alpha \in \Delta$ to be positive if $\alpha(\I X) > 0$. Denote by $\Pi$ the resulting set of simple roots. Then both $\sigma_1$ and $\sigma_2$ permute $\Pi$
		and we may write $\sigma_i = c_i\circ \tau_i$ where $t_i \in N_G(T)$, and where $\tau_i$ is either the identity or the outer automorphism induced by the automorphism of the
		 Dynkin diagram of $\lieg$ corresponding to $\sigma_i|_{\Pi}\colon\Pi\rightarrow\Pi$ (cf.\,\cref{sec-preliminaries}). In fact, $\sigma_i\circ (\tau_i)^{-1}$ permutes 
		 $\Pi$, and since no Weyl group element fixes a Weyl chamber (\cite[theorem 10.3]{Humphreys-lie-algebras}), it follows that $t_i \in T$. Now we distinguish two cases.
		 
		 If $\tau_1 = \tau_2$, then $S$ is also a maximal torus for $K_1$. Indeed, $\sigma_i|_T = \tau_i|_T$, because $t_i \in T$, and thus $\sigma_1|_T = \sigma_2|_T$. Moreover, 
		 we have $T^{\sigma_1} = T_1$ and $(T_1)^{\sigma_2} = S$ by construction, and so in particular $\sigma_2|_{T_1} = \id$. This means that $S = T_1$ is a maximal torus of 
		 the finite-order automorphism $\sigma_1$, and since the pair $(G, K_1)$ is isotropy formal, so is the pair $(G, K)$ by \cref{prop-formality-depends-on-torus}. Next, consider
		 the case $\tau_1 \neq \tau_2$. We claim that either $\tau_1 = \id$ or $\tau_2 = \id$. In fact, if $\lieg$ is not of type $\liealg{d}_4$, then there is at most one non-trivial 
		 automorphism $\Pi\rightarrow\Pi$, see \cite[theorem X.3.29]{Helgason}; and if $\lieg$ is of type $\liealg{d}_4$, then $\Out(G)$ is isomorphic to the permutation group on 
		 three letters and we note that $\tau_1|_{\Pi}$ and $\tau_2|_{\Pi}$ must commute as well, which only is possible if one $\tau_1|_{\Pi}$ or $\tau_2|_{\Pi}$ is the identity 
		 permutation. Thus, one of $\sigma_1$ and $\sigma_2$ is an inner automorphism. Suppose that $\sigma_2$ is inner. Then $\sigma_2$ fixes $T_1$ pointwise, so again 
		 $S = T_1$ and $(G, K)$ is isotropy formal. On the other hand, if $\sigma_1$ is inner, then $T_1 = T$ and $S = T^{\sigma_2}$. But $T_2 = Z_{K_2}(S)$ is a maximal torus for 
		 $K_2$ and $T_2 \subseteq T$. Thus $T_2 = S$, $K$ and $K_2$ share a maximal torus, and $(G, K)$ is isotropy formal. 
    \end{proof}
	
	For the next 
	\namecref{prop-strongly-orthogonal}, which is a generalization of \cite[proposition II.2.2]{thesis},  we also fix a disjoint decomposition $\Omega = \Omega^{+} \cup (-\Omega^{+})$.
	In later sections, when a notion of positivity is fixed on the set of all roots, $\Omega^{+}$ will usually just consist of all positive roots that are contained in $\Omega$, but
	for now $\Omega^{+}$ can be arbitrary. Given $\mu \in \Aut(G)$, we furthermore set $\mu(\alpha) := \alpha\circ\mu^{-1}$ for all roots $\alpha$. With this definition, 
	and assuming that $\mu$ is also an isometry for the bilinear form $\langle\cdot,\cdot\rangle$, we have $\mu(H_{\alpha}) = H_{\mu(\alpha)}$, where $H_{\alpha} \in \liet$ is the vector 
	which is dual to $\alpha$ with respect to $\langle\cdot,\cdot\rangle$.

	\begin{proposition}\label{prop-strongly-orthogonal}
		With notation as above, the following hold
		\begin{enumerate}
			\item 
			$\lieg_{\alpha} \subseteq \liep$ for all $\alpha \in \Omega$, where $\liep \subseteq \lieg$ is the $(-1)$-eigenspace of $\sigma_1$;
			
			\item 
			if $\alpha \in \Omega$, then $\sigma_1\alpha = \alpha$;
			
			\item 
			any two distinct roots of $\Omega$ are strongly orthogonal;
			
			\item 
			we have $\liet = \bigcap_{\alpha\in\Omega^{+}}\ker(\alpha)\oplus\bigoplus_{\alpha\in\Omega^{+}}\CC H_{\alpha}$.		\end{enumerate}
	\end{proposition}
	\begin{proof}
		\hfill\begin{enumerate}
		\item
		Let $V := \lieg_{\alpha} + \lieg_{\sigma_1\alpha}$. Then $V \subseteq Z_{\lieg}(\lies)$, because $\sigma_1$ fixes $\lies$ and $\alpha$ vanishes on $\lies$. Hence 
		\begin{align*}
			(V \cap \liek_1) \subseteq (\liek_1 \cap Z_{\lieg}(\lies)) = Z_{\liek_1}(\lies) = \liet_1.
		\end{align*}
		Therefore, $V \cap \liek_1 = 0$, and since $V$ is $\sigma_1$-invariant, $V \subseteq \liep$.
		
		\item 
		Choose $E_{\alpha} \in \lieg_{\alpha}$, $E_{-\alpha} \in \lieg_{-\alpha}$ with $H_{\alpha} = [E_{\alpha}, E_{-\alpha}]$. By the previous item, and because 
		$[\liep, \liep] \subseteq \liek_1$, we have $H_{\alpha} \in \liek_1$. Thus $H_{\alpha}$ and hence $\alpha$ is fixed by $\sigma_1$.
		
		\item 
		Suppose $\alpha, \beta \in \Omega$ are such that $\alpha + \beta$ is a root. Let $E_{\alpha} \in \lieg_{\alpha}$ and $E_{\beta} \in \lieg_{\beta}$ be nonzero vectors. Again 
		we have, by the first item, that $X = [E_{\alpha}, E_{\beta}] \in \liek_1$. On the other hand, since $G$ is compact, $X$ is also a non-zero eigenvector for $\alpha + \beta$, 
		and $\alpha + \beta$ still vanishes on $\lies$, that is, $\alpha + \beta \in \Omega$. But then also $X \in \liep$ must hold by the first item, which is impossible.
		
		\item 
		Since the elements of $\Omega$ are strongly orthogonal, the vectors $H_{\alpha}$ with $\alpha \in \Omega^{+}$ are mutually perpendicular. In particular, they are linearly
		 independent. Moreover, since $\ad$ is skew-symmetric with respect to $\langle\cdot,\cdot\rangle$, we see that the orthogonal complement of 
		 $\bigoplus_{\alpha\in\Omega^{+}}\CC H_{\alpha}$ is given by $\bigcap_{\alpha\in\Omega^{+}}\ker\alpha$. \qedhere
		\end{enumerate}
	\end{proof}
				
	\begin{proposition}\label{prop-some-power-of-vanishing-root-is-reflected}
		Given $\alpha \in \Omega$, let $\ell \geq 1$ be the largest integer such that the roots $\alpha$, $\sigma(\alpha)$, $\ldots$, $\sigma^{\ell-1}(\alpha)$ are linearly independent
		over $\RR$ (and hence $\CC$). Then $(\sigma_2)^{\ell}(\alpha) = - \alpha$. In particular, no vector in $\bigoplus_{\alpha\in\Omega^{+}} \CC H_{\alpha}$ is
		fixed by $\sigma_2$.
	\end{proposition}
	\begin{proof}
		We  first show that $\liek_2$ intersects $\bigoplus_{\alpha\in\Omega^{+}}\CC H_{\alpha}$ trivially. To this end, note that
		by \cref{prop-strongly-orthogonal} any vector $H_{\alpha}$ with $\alpha \in \Omega^{+}$ is fixed by $\sigma_1$, so $H_{\alpha} \in \liek_1$. Therefore, the projection
		\begin{align*}
			q\colon\lieg &\rightarrow \liek_2, \\
			X &\mapsto \frac{1}{m+1}\sum_{j=0}^m(\sigma_2)^j(X),
		\end{align*}
		where $(m+1)$ is the order of $\sigma_2$, sends $H_{\alpha}$ to $q(H_{\alpha}) \in \liek_1 \cap \liek_2 \cap \liet = \lies$. But for any $Y \in \liealg{s}$ we have 
		\begin{align*}
			\langle Y, q(H_{\alpha})\rangle &= \langle Y, H_{\alpha}\rangle = \alpha(Y) = 0, 
		\end{align*}
		because $\langle\cdot,\cdot\rangle$ is $\sigma_2$-invariant, $\lies$ is fixed pointwise by $\sigma_2$, and $\alpha$ vanishes on $\lies$. Since $\langle\cdot,\cdot\rangle$ is negative-definite on $\lieg_{\RR}$, we see that $q(H_{\alpha}) = 0$. As $q|_{\liek_2} = \id$, it follows that no nontrivial vector is fixed in $\bigoplus_{\alpha\in\Omega^{+}}\CC H_{\alpha}$ by $\sigma_2$.    
		
		On the other hand, because each root $\alpha$ takes purely imaginary values on $\liet_{\RR}$ and $\langle\cdot,\cdot\rangle$ is real valued on $\lieg_{\RR}$, it follows 
		that $H_{\alpha}$ is contained in $\I\lieg_{\RR}$. Thus, $\langle\cdot,\cdot\rangle$ is positive definite on the real span $\bigoplus_{\beta\in\Delta^{+}} \RR H_{\beta}$.
		Now fix a root $\alpha \in \Omega^{+}$. Since $\sigma_2$ is an automorphism of $\lieg_{\RR}$, it hence restricts to an isometry of the real Euclidean vector space 
		\begin{align*}
			V &:= \operatorname{span}_{\RR}\left\{ (\sigma_2)^j(H_{\alpha}) \,|\, j \in \ZZ_{\geq 0} \right\}.
		\end{align*}
		Set $v_i := (\sigma_2)^i(H_{\alpha})$ and recall that $\ell$ is the maximal integer such that $v_0, \ldots, v_{\ell - 1}$ are linearly independent and hence form a basis of $V$.
		In fact, each vector $v_i$ equals $H_{\beta}$ for some root $\beta \in \Omega$, because the set $\Omega$ is $\sigma_2$ invariant, and since $v_i \neq \pm v_j$,
		the vectors $v_0, \ldots, v_{\ell - 1}$ must be even mutually perpendicular by \cref{prop-strongly-orthogonal}. Thus, if we express $(\sigma_2)^{\ell}(H_{\alpha})$ as 
		$(\sigma_2)^{\ell}(H_{\alpha}) = c_0v_0 + \ldots + c_{\ell-1}v_{\ell-1}$ for real numbers $c_0, \ldots, c_{\ell-1}$, then we may compute 
		\begin{align*}
			\langle H_{\alpha}, H_{\alpha}\rangle &= \langle (\sigma_2)^{\ell}(H_{\alpha}), (\sigma_2)^{\ell}(H_{\alpha})\rangle = \sum_{t=0}^{\ell-1} |c_t|^2\cdot \langle H_{\alpha}, H_{\alpha}\rangle.
		\end{align*}
		This gives $\sum_{t=0}^{\ell-1}|c_t|^2 = 1$. In terms of the basis $(v_i)_{i=0,\ldots, \ell-1}$ the endomorphism $\sigma_2|_V$ has matrix
		\begin{align*}
			\begin{pmatrix}
		 		0 & \ldots  & \ldots    & 0         & c_0   \\
		 		1 & \ddots  &           & \vdots    & c_1   \\
		 		0 & 1       & \ddots    & \vdots    & \vdots \\
		 		\vdots & \ddots  & \ddots    & 0         & c_{\ell-2} \\
				0  & \ldots      & 0          & 1 & c_{\ell-1}
			\end{pmatrix},
		\end{align*}
		which shows that $\det(\sigma_2|_V) = \pm c_0$. But since $\sigma_2|_V$ also is an isometry of an Euclidean space, we must have $c_0 = \pm 1$ and thus we deduce from 
		the equation $|c_0|^2 + \ldots + |c_{\ell-1}|^2 = 1$ that $c_1, \ldots, c_{\ell-1}$ must all vanish identically. We claim that actually $c_0 = -1$. Indeed, if $c_0 = 1$, then
		$\sigma_2(v_{\ell-1}) = v_0$ and $v_0 + \ldots + v_{\ell-1}$ would be a non-zero fixed point for $\sigma_2$ in $\bigoplus_{\alpha\in\Omega^{+}} \CC H_{\alpha}$, which, as we 
		already observed earlier, is impossible. Therefore, $c_0 = -1$ and $(\sigma_2)^{\ell}(\alpha) = -\alpha$.
	\end{proof}	
	
	The proof of \cite[proposition II.3.2]{thesis} carries over verbatimely to show:
	
	\begin{corollary}\label{cor-estimate-on-rank}
		We have $\dim\lies \leq \rank\liek_1 - |\Omega^{+}|$, with equality if $\sigma_2$ is inner.
	\end{corollary}
	\begin{proof}
		By \cref{prop-strongly-orthogonal}, $\sigma_1$ fixes $V := \bigoplus_{\alpha\in\Omega^{+}} \CC H_{\alpha}$ pointwise, while no vector in $V$ is fixed by $\sigma_2$. 
		Setting $L := \bigcap_{\alpha\in\Omega^{+}} \ker \alpha$, we also note that the same \namecref{prop-strongly-orthogonal} gives the $\sigma_2$-invariant
		decomposition $\liet_1 = (L \cap \liet_1)\oplus V$
		which, as $\lies = (\liet_1)^{\sigma_2}$, then yields $\dim\lies \leq \rank\liek_1 - |\Omega^{+}|$. 
		
		If $\sigma_2$ is an inner automorphism, say with $\sigma_2 = c_n$ for 
		some element $n \in G$, then because $\sigma_2$ fixes $S$ pointwise, we see that $n \in Z_G(S)$. Since centralizers of tori are connected \cite[corollary IV.4.51]{Knapp},
		we may write $n = \exp(X)$ for some element $X \in Z_{\lieg}(\lies)$, and by \cref{prop-centralizers-of-subtori} we can decompose $X$ as
		$X = X_0 + X_{\Omega}$ for certain elements $X_0 \in \liet$ and $X_{\Omega} \in \bigoplus_{\alpha\in\Omega} \lieg_{\alpha}$. In particular, if
		$Y \in L$, then $[Y, X] = 0$. Thus $\Ad_n$ fixes $L \subseteq \liet$ pointwise. 
		As already remarked, no vector in $V$ is fixed by $\sigma_2$, so the fixed point set of $\sigma_2$ on $\liet$ is
		precisely $L$. Then $\lies = L \cap \liet_1$.
	\end{proof}	
		
	\begin{definition}\label{def-property-reflective}
		With the notation as in the paragraph preceeding \cref{prop-strongly-orthogonal}, we say that $\sigma_1$, $\sigma_2$ satisfy property $\star$ if $\sigma_2(\alpha) = -\alpha$ 
		holds for all $\alpha \in \Omega$.
	\end{definition}
	
	\begin{theorem}\label{thm-normal-form-automorphisms}
		Let $\Delta^{+}$ be a choice of positive roots on $\liet$ and $\Omega^{+} = \Omega \cap \Delta^{+}$. If $\sigma_1$ and $\sigma_2$ satisfy property $\star$, then there exists 
		$w \in N_G(T)$ with the following properties.
		\begin{enumerate}
		\item 
		$w\sigma_2w^{-1} = c_g\circ\nu$, where $\nu$ is an automorphism of $G$ with $\nu(\Delta^{+}) = \Delta^{+}$,
		
		\item 
		each root $\alpha \in \tilde{\Omega}$, where $\tilde{\Omega} := w(\Omega)$, is fixed by $\nu$, and 
		
		\item 
		$\Ad_g|_{\liet} = \prod_{\alpha\in\tilde{\Omega}^{+}} s_{\alpha}$.
		\end{enumerate}
	\end{theorem}
	\begin{proof}
		The proof is by induction on the dimension of $G$. If $G$ is a compact, connected Lie group of dimension $1$,
		then $G \cong S^1$ as Lie groups. In this case, there are no roots on $G$ and the statement of the theorem is vacuous. This establishes the induction base.
		Suppose now that the statement of the theorem is true for all compact, connected Lie groups of dimension at most $r$ and let $G$ be a compact, connected Lie group of dimension $r + 1$. We distinguish the cases $\Omega = \emptyset$ and $\Omega \neq \emptyset$. Suppose first that $\Omega$ is empty. This means that no $\lieg$-root vanishes on $\lies$, and hence $\lies$ contains regular elements for the $\liet$-roots. We fix one such regular element and denote the resulting notion of positivity by $R^{+}$. Since $G$ is compact, there exists a Weyl group element $w$ such that $w(R^{+}) = \Delta^{+}$, and then $w\sigma_2 w^{-1}$ preserves $\Delta^{+}$. Thus, the induction step follows in this case.
		Now assume that $\Omega$ is nonempty and fix one root $\delta \in \Omega$. Since the orbit $W(G, T)\cdot \delta$ of the Weyl group $W(G, T)$ contains a dominant root by \cite[corollary 2.68]{Knapp}, we may assume, conjugating $\sigma_2$ by such a Weyl group element if necessary, that $\delta$ is dominant, that is, that $\langle\delta,\alpha\rangle \geq 0$ holds for all $\alpha \in \Delta^{+}$. Now express $\sigma_2$ as $\sigma_2 = c_x\circ\rho$, where $x \in N_G(T)$ and where $\rho$ is an automorphism of $G$ with $\rho(\Delta^{+}) = \Delta^{+}$. 
		
		\begin{claim}\label{thm-normal-form-automorphisms-claim-dynkin-auto-fixes-dominant-root}
		$\rho(\delta) = \delta$. 
		\end{claim}
		\begin{proofofclaim}
			Indeed, by assumption we have $-\delta = \sigma_2(\delta)$, and hence also
			\begin{align*}
				-\Ad_{x^{-1}}(\delta) &= \rho(\delta);
			\end{align*}
			now the set of $\liet$-roots $\Delta$ decomposes as a disjoint union of mutually perpendicular, irreducible components $\Delta = \Delta_1 \cup \ldots \cup \Delta_m$. Since $\Ad_x$ is a Weyl group element, it maps each such irreducible component into itself, and hence $\delta$ and $-\Ad_{x^{-1}}(\delta)$ are contained in the same irreducible component. Moreover, since $\rho$ just permutes the simple roots and because $\rho$ is an isometry with respect to $\langle\cdot,\cdot\rangle$ (because $\sigma_2$ and $\Ad_x$ are), we see that $\rho(\delta)$ is a dominant root which is contained in the same irreducible component as $\delta$. However, in an irreducible root system there exist at most two dominant roots, namely the highest long and, if there are two root lengths, the highest short root, and since $\rho$ preserves lengths, we conclude that $\rho(\delta) = \delta$.
		\end{proofofclaim}
		
		This means that $\Ad_x(\delta) = -\delta$, so $s_{\delta}\circ(\Ad_x|_{\liet})$ is a Weyl group element preserving $\delta$. By Chevalley's Lemma (cf. \cite[proposition 2.72]{Knapp}), $\Ad_x|_{\liet}$ then is the product of $s_{\delta}$ and certain root reflections $s_{\beta}$ with $\langle\beta,\delta\rangle = 0$. Let $\Delta' \subseteq \Delta$ be the set of all roots $\beta$ with $\langle\beta,\delta\rangle = 0$. Since $\sigma_1$ and $\sigma_2$ map $\delta$ to $\pm\delta$ by \cref{prop-strongly-orthogonal}, we conclude that $\sigma_1$ and $\sigma_2$ leave $\Delta'$ invariant. Hence, they also leave invariant the subalgebra
		\begin{align*}
		\lieg' &= \liet'\oplus\bigoplus_{\alpha\in\Delta'}\lieg_{\alpha}\text{, where }\liet = \CC H_{\delta}\oplus\liet'\text{ and }\liet' = \ker \delta.
		\end{align*}
		Note that, because $\lieg$ is the complexification of $\lieg_{\RR}$, also $\lieg'$ is the complexification of the real subalgebra $\lieg'_{\RR} := \lieg' \cap \lieg_{\RR}$. Let $G' \subseteq G$ be the connected Lie subgroup of $G$ with Lie algebra $\lieg'_{\RR}$. 
		
		\begin{claim}\label{thm-normal-form-automorphisms-claim-subgroup-is-compact}
		\textit{$G'$ is compact.}
		\end{claim}
		\begin{proofofclaim}
			Indeed, the subalgebra $\lieg'' = \liet''\oplus\bigoplus_{\alpha\in\Delta'}\lieg_{\alpha}$ with $\liet'' = \sum_{\alpha\in\Delta'} \CC H_{\alpha}$ is semisimple, hence if we let $G'' \subseteq G$ be the connected subgroup corresponding to $\lieg''$, then $G''$ is compact. We also note that the connected Lie subgroup $T' \subseteq T$ with Lie algebra $\liet' \cap \lieg_{\RR}$ is a torus, because it is the fixed point set of $s_{\delta}$ on $\liet \cap \lieg_{\RR}$ and because $s_{\delta}$ is the restriction of some inner automorphism of $\lieg$ to $\liet$. Hence, $T'\cdot G'' \subseteq G$ is a compact, connected Lie group containing $G'$. But $T'\cdot G'' \subseteq G'$, and so it follows that $T'\cdot G'' = G'$.
		\end{proofofclaim}
		
		Write $\sigma_i'$ for the automorphism $G'\rightarrow G'$ obtained by restricting $\sigma_i$ to $G'$ and let $K_i'$ be the identity component of the fixed point set of 
		$\sigma_i'$.
		
		\begin{claim}\label{thm-normal-form-automorphisms-claim-maximal-torus}
		$S$ is a maximal torus of $K_1'\cap K_2'$. The set of $\liet'$-roots vanishing on $\lies$ is precisely given by the set of all elements of the form $\alpha|_{\liet'}$ with $\alpha \in \Omega - \{\pm\delta\}$.
		\end{claim}
		\begin{proofofclaim}
			Because $\delta \in \Omega$, we know that $\lies \subseteq \ker\delta$, and therefore $\lies \subseteq \liek_1' \cap \liek_2'$. Since $\liek_1' \cap \liek_2'$ is a subalgebra of $\liek_1\cap\liek_2$, and since by assumption $\lies$ is a maximal torus of the latter algebra, $\lies$ also is a maximal torus of the former algebra. Now note that, by construction, 
			the $\liet'$-roots are precisely the restrictions of roots in $\Delta'$ to $\liet'$. Hence, if a root $\beta$ in $\Delta'$ vanishes on $\lies$, then $\beta = \alpha|_{\liet'}$ for some $\liet$-root $\alpha$ and $\alpha$ vanishes on $\lies$. That is, $\alpha \in \Omega$. Since $\delta \not\in \Delta'$, we must have $\alpha \in \Omega-\{\pm\delta\}$.
		\end{proofofclaim}
		
		Let $T' \subseteq T$ be the subtorus corresponding to $\liet'$ and $T_1' = Z_{K_1'}(S)$. We have $T_1' \subseteq Z_{K_1}(S) \cap G'$ and so $T_1' \subseteq T \cap G'$. 
		It is immediate that $T' = T \cap G'$, whence $T_1' \subseteq T'$ and $T'$ is a maximal torus of $G'$ containing $T_1'$. As $Z_{G'}(T_1')$ is the unique 
		maximal torus in $G'$ containing $T_1'$, we thus see that $T' = Z_{G'}(T_1')$. Therefore, if we endow the maximal torus $T'$ of the compact connected Lie group $G'$ with the notion of positivity $(\Delta' \cap \Delta^{+})|_{\liet'}$, then the induction hypothesis applies to $\sigma_1'$ and $\sigma_2'$. We hence find a Weyl group element $w \in N_{G'}(T')$ such that $w\sigma_2'w^{-1} = \Ad_h\circ\mu$, such that $h \in G'$ corresponds to the product of the root reflections $s_{\beta}$ with $\beta \in w(\Omega')$, and such that $\mu$ permutes the positive roots on $\liet'$ and fixes the elements of $w(\Omega')$. Note that since the Weyl group of $\lieg'$ is generated by the root reflections $s_{\beta}$ with $\beta \in \Delta'$, we have $N_{G'}(T') \subseteq N_G(T)$ and $w(\delta) = \delta$. Hence we can express $w\sigma_2w^{-1}$ as $w\sigma_2w^{-1} = \Ad_g\circ \nu$ for some automorphism $\nu$ with $\nu(\Delta^{+}) = \Delta^{+}$ and some element $g \in N_G(T)$, and moreover $(\Ad_g\circ\nu)(\delta) = -\delta$. The same proof as in \cref{thm-normal-form-automorphisms-claim-dynkin-auto-fixes-dominant-root} now shows that $\nu(\delta) = \delta$ and $\Ad_g(\delta) = -\delta$, because $w\sigma_2w^{-1}$ is an isometry. In particular, $\Ad_g$ and $\nu$ preserve $\lieg'$, and so 
		\begin{align*}
		c_g|_{G'}\circ \nu|_{G'} &= w\sigma_2'w^{-1} = c_h\circ\mu;
		\end{align*}
		but both $\nu|_{G'}$ and $\mu$ preserve the notion of positivity on $\liet'$, and hence so must do $c_{h^{-1}}\circ c_g|_{G'}$. On the other hand, using Chevalley's Lemma 
		again it follows that
		$g = g_1 g_2$ where $g_2 \in N_G(T)$ corresponds to $s_{\delta}$ and $g_1 \in G'$ corresponds to a certain product of root reflections $s_{\beta}$ with $\beta \in \Delta'$.
		Since $s_{\delta}$ fixes $\liet'$ pointwise and no non-trivial Weyl group element fixes a Weyl chamber, we see that $g_1T' = hT'$ and that $\nu|_{\liet'} = \mu|_{\liet'}$. In particular, 
		\begin{align*}
		\Ad_g|_{\liet'} &= \Ad_h|_{\liet'} = \prod_{\alpha\in w(\Omega') \cap \Delta^{+}} s_{\alpha}
		\end{align*}
		But $\Ad_g(\delta) = -\delta$ and $w(\delta) = \delta$, so $w(\Omega) = w(\Omega') \cup \{\pm\delta\}$ and 
		hence $\Ad_g|_{\liet} = \prod_{\alpha} s_{\alpha}$, where $\alpha$ now ranges over the elements of $w(\Omega) \cap \Delta^{+}$. Since $\nu$ fixes $\delta$ and $\mu$ fixes $w(\Omega')$, we conclude that $\nu$ fixes the elements of $w(\Omega)$, and thus the induction claim is proved.
	\end{proof}
	
	We will have use for yet another generalization of a proposition that was proven for $\ZZ_2\times\ZZ_2$-symmetric spaces in \cite{thesis}, cf.\,\cite[proposition II.3.5]{thesis}.
	\begin{corollary}\label{cor-inner-automorphisms-normal-form-for-roots-vanishing-on-maximal-torus}
		Suppose that $\sigma_1$ is an inner automorphism and that $\sigma_2|_{\liet}$ is an involution. There exists a notion of positivity $\Delta^{+}$ 
		such that the conclusions of \cref{thm-normal-form-automorphisms} hold with $w = \id$. Moreover, if $\beta$ is a root with $\nu(\beta) = \beta$, then the integer
		\begin{align*}
			\sum_{\alpha\in\Omega^{+}} \frac{2\langle \alpha, \beta\rangle}{\langle\alpha,\alpha\rangle}
		\end{align*}
		is even. If $\sigma_2$ is an inner automorphism as well, then this integer is even for all roots $\beta$.
	\end{corollary}
	\begin{proof}
		Start with an arbitrary notion of positivity $R^{+}$ on $\Delta$. Because $\sigma_2|_{\liet}$ is an involution, $\sigma_1$ and $\sigma_2$ satisfy property $\star$.
		Hence, by \cref{thm-normal-form-automorphisms} there exists a Weyl group element $w$ such that $w\sigma_2w^{-1} =  c_k\circ\mu$ and such that $\Ad_k$, $\mu$ are
		as in \cref{thm-normal-form-automorphisms}. Then $\sigma_2 = c_g\circ \nu$, where $g = w^{-1}(k)$ and $\nu = w^{-1}\mu w$. Note that $\Ad_g|_{\liet}$ represents 
		$\prod_{\alpha\in\Omega^{+}} s_{\alpha}$ and that $\mu$ fixes each element in $\Omega^{+}$. Moreover, $\mu$ permutes the roots in $\Delta^{+} := w^{-1}(R^{+})$,
		and so the conclucions of \cref{thm-normal-form-automorphisms} hold for the notion of positivity $\Delta^{+}$.
	
		Use \cref{prop-strongly-orthogonal} to decompose $\liet$ as $\liet = L\oplus V$ with $V = \bigoplus_{\alpha\in\Omega^{+}}\CC H_{\alpha}$ and 
		$L = \bigcap_{\alpha\in\Omega^{+}} \ker\alpha$. We know that $\nu$ fixes each element in $\Omega^{+}$. Thus, this decomposition is a decomposition
		into $\Ad_g$- and $\nu$-invariant subspaces. In fact, $L$ is precisely the $1$-eigenspace and $V$ is precisely the $(-1)$-eigenspace of $\Ad_g|_{\liet}$,
		and we may further decompose $L$ as $L = L^{+}\oplus L^{-}$, where $L^{+}$ and $L^{-}$ are, respectively, the $1$- and $(-1)$-eigenspaces of 
		$\nu$ on $L$.
		
		Because $\sigma_1$ is an inner automorphism fixing $T$ pointwise, we must have $\sigma_1 = c_t$ with $t \in T$. Decompose now $t$ as
		$t = t_{+}t_{-}$, where $t_{+}$ is contained in the image of $L^{+}$ under the exponential map and $t_{-}$ is contained in the image of $L^{-}\oplus V$.
		Then $\sigma_2(t_{-}) = (t_{-})^{-1}$ and $\sigma_2(t_{+}) = t_{+}$. Since $\sigma_2$ and $\sigma_1$ commute, this yields
		\begin{align*}
			c_{t_{+}}\circ c_{t_{-}} &= \sigma_1 = \sigma_2\circ\sigma_1\circ(\sigma_2)^{-1} = c_{t_{+}}\circ (c_{t_{-}})^{-1},
		\end{align*}
		and so $c_{t_{-}}$ must be an involution. Let $\beta \in \Omega$ be arbitrary. By \cref{prop-strongly-orthogonal} $\sigma_1$ acts as $-\id$ on $\lieg_{\beta}$ for all
		$\beta \in \Omega$, while by construction $\Ad_{t_{+}}$ is the identity on each such root space. Therefore, $\Ad_{t_{-}}|_{\lieg_{\beta}} = -\id$ must hold as well.
		Hence, if we write $t_{-}$ as $t_{-} = \exp(X + \I\pi\sum_{\alpha\in\Omega^{+}} \frac{m_{\alpha}}{\langle\alpha,\alpha\rangle}H_{\alpha})$ for 
		certain real coefficients $m_{\alpha}$ and an element $X \in L^{-}$, and if we recall that the elements of $\Omega$ are mutually perpendicular by 
		\cref{prop-strongly-orthogonal}, then we see that 
		\begin{align*}
			-\id &= \Ad_{t_{-}}|_{\lieg_{\beta}} = \mathrm{e}^{\I\pi m_{\beta}}.
		\end{align*} 
		This shows that $m_{\beta}$ must be an odd number. Similarly, writing out the equation $\id = (\Ad_{t_{-}})^2|_{\lieg_{\beta}}$, where $\beta$ now is an
		arbitrary root, we conclude that 
		\begin{align*}	
			\sum_{\alpha\in\Omega^{+}} m_{\alpha}\frac{2\langle\alpha,\beta\rangle}{\langle\alpha,\alpha\rangle} - \I\beta(X)
		\end{align*}
		must be an even number. If $\beta$ is a root with $\nu(\beta) = \beta$, then $\beta(X) = 0$, because $X \in L^{-}$, and since $\frac{2\langle\alpha,\beta\rangle}{\langle\alpha,\alpha\rangle}$ is always an integer and $m_{\alpha}$ is odd, the assertion follows.
	\end{proof}

	\section{Property $\star$ on simple Lie groups}\label{sec-reflection}
	
	As in the previous section, we let $G$ be a compact, connected Lie group and we assume that we are given an involutive automorphism $\sigma_1$ as well as a finite-order order 
	automorphism $\sigma_2$ on $G$ which commutes with $\sigma_1$. The identity component of $G^{\sigma_1} \cap G^{\sigma_2}$ is denoted $K$, the identity component of 
	$G^{\sigma_1}$ is denoted $K_1$, and $S \subseteq K$ is an arbitrary maximal torus. Recall that $T_1 = Z_{K_1}(S)$ is a maximal torus in $K_1$ and that $T = Z_G(T_1)$
	is a maximal torus in $G$. We let $\Delta$ be the $G$-roots with respect to $\liet$ and let $\Omega$ be the roots restricting to zero on $\lies$. Within this section, we
	also fix an auxiliary notion of positivity $\Delta^{+}$ on $\Delta$. With this notation, the goal of this section is to prove the following generalization of \cite[proposition II.2.2]{thesis}.
	
	\begin{theorem}\label{thm-reflection}
		If $G$ is simple, then for any root $\alpha \in \Omega$ we have $\sigma_2(\alpha) = -\alpha$.
	\end{theorem}
	
	If one assumes that $\sigma_2|_{\liet}$ is an involution, then this is a straightforward consequence of \cref{prop-some-power-of-vanishing-root-is-reflected}. The
	general case, however, is more involved and will occupy the remainder of this section. Curiously enough, on simple Lie groups $G$ whose isomorphism type is not
	$\liealg{d}_4$, \cref{thm-reflection} and \cref{thm-normal-form-automorphisms} \emph{a fortiori} show that $\sigma_2|_{\liet}$ is an involution.

	\begin{lemma}\label{lem-order-on-roots}
		Let $\alpha \in \Omega$ and let $j \geq 1$ be the smallest integer such that $(\sigma_2)^j(\alpha) = \alpha$. Then the integer $\ell \geq 1$ figuring in
		\cref{prop-some-power-of-vanishing-root-is-reflected} satisfies $j = 2\ell$. 
	\end{lemma}
	\begin{proof}
		Indeed, we have $(\sigma_2)^{2\ell}(\alpha) = \alpha$, whence $2\ell \geq j$. 
		Next, note that for any non-negative integer $r$ with $(\sigma_2)^r(\alpha) = -\alpha$ we must have $r \geq \ell$, for otherwise 
		$\alpha, \sigma_2(\alpha), \ldots, (\sigma_2)^{\ell - 1}(\alpha)$ would not be linearly independent. In particular, we must have $\ell < j$,
		as otherwise we could express $\ell$ as $\ell = mj + r$ with integers $m, r \geq 0$ and $r < \ell$, which would imply that also $(\sigma_2)^r(\alpha) = -\alpha$. 
		A similar computation shows that $2\ell = j$. In fact, if we write $j = m\ell + r$ for certain integers $m, r \geq 0$ with $r < \ell$, and assume that $2\ell > j$ for a contradiction, then 
		$m = 1$ and we find that
		\begin{align*}
			\alpha &= (\sigma_2)^j(\alpha) = (\sigma_2)^r\bigl((\sigma_2)^{\ell}(\alpha)\bigr) = -(\sigma_2)^r(\alpha).
		\end{align*}
		But this means that $(\sigma_2)^r(\alpha) = -\alpha$, which again is impossible.
	\end{proof}
	
	\begin{lemma}\label{lem-maximal-torus}
		Let $j \geq 0$ and put $\tau := (\sigma_2)^j$. Then $\liet_1$ contains a maximal torus of $\liek_1 \cap \lieg^{\tau}$.
	\end{lemma}
	\begin{proof}
		Because every point which is fixed by $\sigma_2$ also is fixed by $\tau$, we have $\lies \subseteq \liek_1 \cap \lieg^{\tau}$. Therefore, there
		exists a maximal torus $\lieh_{\RR}$ of $\liek_1 \cap (\lieg_{\RR})^{\tau}$ containing $\lies_{\RR}$. On the other hand, $\tau$ restricts to a finite-order automorphism 
		$\tau|_{\liek_1}\colon \liek_1 \rightarrow \liek_1$ of $\liek_1$, and so, as already observed earlier, $Z_{\liek_1}(\lieh)$ must be a maximal torus of $\liek_1$. But 
		$\lies \subseteq Z_{\liek_1}(\lieh)$ and $\liet_1$ is the unique maximal torus of $\liek_1$ containing $\lies$, whence we must have $Z_{\liek_1}(\lieh) = \liet_1$, and 
		in particular $\lieh \subseteq \liet_1$.
	\end{proof}
	
	\begin{lemma}\label{lem-order-at-most-four-on-vanishing-roots}
		Denote by $\liealg{q} = [\liek_1, \liek_1]$ the semisimple part of $\liek_1$ and assume that $\Omega$ is non-empty. Then, on a simple Lie group $G$, exactly one of the following cases occurs.
		\begin{enumerate}
			\item
			$\liealg{q}$ contains a $\sigma_2$-invariant simple ideal. In this case, $\sigma_2(\alpha) = -\alpha$ for all $\alpha \in \Omega$.
			
			\item 
			$\liealg{q} = \lieh_1\oplus\lieh_2$ is a sum of simple ideals $\lieh_1$ and $\lieh_2$ which are mapped isomorphically onto each other by $\sigma_2$. Furthermore, 
			we have $\sigma_2(\alpha) = -\alpha$ or $(\sigma_2)^2(\alpha) = -\alpha$ for all $\alpha \in \Omega$.
		\end{enumerate}
	\end{lemma}
	\begin{proof}
		Fix $\alpha \in \Omega^{+}$. We know from \cref{prop-strongly-orthogonal} that $\alpha$ is fixed by $\sigma_1$, so $H_{\alpha}$ is contained in the maximal torus
		$\liet_1$ of $\liek_1$. As $K_1$ is a compact Lie group, $\liek_1$ splits as $\liek_1 = Z\oplus \liealg{q}$, where $Z$ is the (possibly trivial) center of $\liek_1$,
		so we may decompose $H_{\alpha}$ accordingly as $H_{\alpha} = H_0 + H_s$, where $H_0 \in Z$ and $H_s \in \liealg{q}$.
		
		\begin{claim}\label{lem-maximal-torus-claim-1}
			If $\lieh \subseteq \liealg{q}$ is a $\mu$-invariant simple ideal, $\mu  = (\sigma_2)^j$, and $\mu|_{\lieh}$ is an inner automorphism, then $\mu$ fixes 
			$\liet_1 \cap \lieh$ pointwise.
		\end{claim}
		\begin{proofofclaim}
			By assumption, $\rank \lieh^{\mu} = \rank \lieh$. Since $\liet_1 \cap \lieh$ contains a maximal torus of $\lieh^{\mu}$  by \cref{lem-maximal-torus}
			we see that $\mu$ must be the identity on $\liet_1 \cap \lieh$.
		\end{proofofclaim}
		
		\begin{claim}\label{lem-maximal-torus-claim-2}
			If $H_0$ is non-zero, then $\sigma_2(H_0) = -H_0$. In particular, $(\sigma_2)^{\ell}(\alpha) = -\alpha$ only if $\ell$ is odd.
		\end{claim}
		\begin{proofofclaim}
			Indeed, by the classification of fixed point sets of involutions (\cref{prop-classification-of-involutions}) the center $\liek_1$ can be at most one-dimensional, and since 
			$\sigma_2\colon\liek_1\rightarrow\liek_1$ is a finite-order automorphism induced by the $\RR$-linear map 
			$\sigma_2\colon\lieg_{\RR}\rightarrow\lieg_{\RR}$, we conclude that $H_0$ must necessarily be an eigenvector of $\sigma_2$ to the eigenvalue $1$ or $-1$. But
			we know from \cref{prop-some-power-of-vanishing-root-is-reflected} that $(\sigma_2)^{\ell}(\alpha) = -\alpha$ for some $\ell \geq 1$, and so $H_0$ must be an
			eigenvector to the eigenvalue $-1$.
		\end{proofofclaim}
		
		\begin{claim}
			If there exists a simple ideal $\lieh \subseteq \liealg{q}$ which is $\sigma_2$-invariant, then $\sigma_2(\alpha) = -\alpha$.
		\end{claim} 
		\begin{proofofclaim}
			First note that either $(\sigma_2)^2|_{\lieh}$ or $(\sigma_2)^3|_{\lieh}$ must fix $\liet_1 \cap \lieh$ pointwise. In fact, if $\mu\colon L \rightarrow L$ is a Lie group 
			automorphism on a compact simple Lie group $L$, then either $\mu^2$ or $\mu^3$ is an inner automorphism, since the outer automorphism group
			$\Out(L) = \Aut(L)/\mathrm{Int}(L)$ is isomorphic to the automorphism group of the Dynkin diagram of $\liealg{l}$, which only admits automorphisms of
			order at most $3$ \cite[theorem X.3.29]{Helgason}.		
			Applying this observation to $\sigma_2|_{\lieh}$, we see that $(\sigma_2)^j|_{\lieh}$ must be an inner automorphism for $j = 2$ or $j = 3$, so $(\sigma_2)^j$ fixes
			$\liet_1 \cap \lieh$ pointwise by \cref{lem-maximal-torus-claim-1}. Moreover, if $\liealg{q}$ is not simple, then by the 
			\cref{prop-classification-of-involutions}, we know that $\liealg{q}$ is 
			the sum of two simple ideals, whence the other simple ideal $\lieh'$ of $ \liealg{q}$ must be $\sigma_2$-invariant as well. The same reasoning then shows that 
			$(\sigma_2)^2$ or $(\sigma_2)^3$ must be the identity on $\liet_1 \cap \lieh'$. Now recall that $H_{\alpha} = H_0 + H_s$ with $H_s \in \liealg{q}$ an $H_0$ an
			element in the center of $\liek_1$. We decompose $H_s$ further as $H_s = H + H'$, where $H \in \lieh$, 
			$H' \in \lieh'$, and denote by $j \geq 1$ the minimal number with $(\sigma_2)^j(\alpha) = \alpha$. We know from \cref{lem-order-on-roots} that $j = 2\ell$ for some 
			$\ell \geq 1$ with $(\sigma_2)^{\ell}(\alpha) = -\alpha$, so to prove our claim it suffices to show that $j \leq 3$. In order to do this, we distinguish two cases.
			
			First, suppose that $H_0$ is non-zero. According to \cref{lem-maximal-torus-claim-2} the integer $\ell$ then must be odd. On the other hand, since $(\sigma_2)^6$
			is the identity on $\liet_1\cap\lieh$ and $\liet_1\cap\lieh'$, it follows that $j \leq 6$, and hence we either have $j = 2$ or $j = 6$. However, if we were to assume that $j = 6$, then
			one of $H$ or $H'$ must be nonzero, because $\sigma_2$ is of order $2$ on $\CC\cdot H_0$. Also $(\sigma_2)^3(\alpha) = -\alpha$. Hence, if
			$H$ is nonzero, then $\sigma_2$ must be of order at most $2$ on $\liet_1 \cap \lieh$, and similarly, if $H'$ is nonzero, then $\sigma_2$ must be of order at most $2$
			on $\liet_1 \cap \lieh'$. But then $j \leq 2$, a contradiction.
			
			Next, assume that $H_0 = 0$. If one of $H$ or $H'$ is zero, then $j \leq 3$ is immediate from the fact that $\sigma_2$ is of order at most $3$ on $\liet_1 \cap \lieh$ and
			$\liet_1 \cap \lieh'$. Otherwise, let $r, s \in\{2, 3\}$ be the minimal integers such that $(\sigma_2)^r(H) = H$ and $(\sigma_2)^s(H') = H'$,
			and note that $j = \mathrm{lcm}(r, s)$. Hence, if $r = s$, then it again follows that $j \leq 3$. Suppose then that $r$ and $s$ are different, for a contradiction. Since
			$r, s \in \{2, 3\}$, this is only possible if $j = 6$. But then we would have $(\sigma_2)^3(\alpha) = -\alpha$, contradicting the assumptions $(\sigma_2)^r(H) = H$ and 
			$(\sigma_2)^s(H') = H'$.
		\end{proofofclaim}
		
		\begin{claim}
			If $\liealg{q}$ contains no $\sigma_2$-invariant simple ideal, then $\sigma_2(\alpha) = -\alpha$ or $(\sigma_2)^2(\alpha) = -\alpha$.
		\end{claim}
		\begin{proofofclaim}[\qedhere]
			Appealing once more to \cref{prop-classification-of-involutions}, we find that $\liealg{q} = \lieh_1\oplus \lieh_2$ must be a sum of two simple ideals 
			such that $\sigma_2(\lieh_1) = \lieh_2$. This means that $\lieh_1$ is an invariant simple ideal of $\mu = (\sigma_2)^2$, which implies that either $\mu|_{\lieh_1}$, 
			$\mu^2|_{\lieh_1}$, or $\mu^3|_{\lieh_1}$ is an inner automorphism. Because $\sigma_2\colon \lieh_1 \rightarrow \lieh_2$ is an isomorphism of Lie algebras, the 
			orders of $\sigma_2$ on $\liet_1 \cap \lieh_1$ and $\liet_1 \cap \lieh_2$ must agree, and so we conclude, using \cref{lem-maximal-torus-claim-1}, that $\mu$ is of order at most $3$ on 
			$\liet_1 \cap \liealg{q}$. It still remains to show that $(\sigma_2)^4(\alpha) = \alpha$, for then $\sigma_2(\alpha) = -\alpha$ or $(\sigma_2)^2(\alpha) = -\alpha$ has to hold by 
			\cref{lem-order-on-roots}. Assume that this was not the case for a contradiction. That is, assume that $\mu^2(\alpha) \neq \alpha$ and in particular $\mu(\alpha) \neq -\alpha$
			holds. This implies that $\mu^2|_{\lieh_1}$ cannot be an inner automorphism. 
			Indeed, $\lieh_1$ is a $\mu$-invariant simple ideal, so if $\mu^2|_{\lieh_1}$ was an inner automorphism, then $\mu^2$ would fix $\liet_1 \cap \lieh_1$ pointwise by
			\cref{lem-maximal-torus-claim-1}. But then $\mu^2$ would fix $\liet_1 \cap \lieh_2$ pointwise as well, and since according to \cref{lem-maximal-torus-claim-2} $\mu$ always fixes $H_0$, $\mu$ would also have to fix $\alpha$, which we have assumed to 
			be not the case. Therefore, $\mu|_{\lieh_1}$
			must be an element of order $3$ in $\Out(\lieh_1)$, which is only possible if the isomorphism type of $\lieh_1$ is $\liealg{d}_4$. Similarly, we see that 
			$\lieh_2$ must be of type $\liealg{d}_4$. Now it follows from \cref{prop-classification-of-involutions} that $\lieg$ is of type $\liealg{d}_8$, that $\sigma_1$ is 
			an inner automorphism of $\lieg$, and that $\liek_1 = \liealg{q}$ is semisimple. In particular, any element in $\Out(\lieg)$ is of order $2$, whence 
			$\mu = (\sigma_2)^2$ must be an inner automorphism of $\lieg$. 
			
			With these observations we can now compute $|\Omega^{+}|$ in two different ways. First, note that the map $(\lieh_1)^{\mu} \rightarrow (\liek_1)^{\sigma_2}$, 
			$X \mapsto X + \sigma_2(X)$, is an isomorphism, and hence $\rank\, (\lieh_1)^{\mu} = \dim\lies$. Because $\mu|_{\lieh_1}$ is an outer automorphism of order $3$
			on a Lie algebra of isomorphism type $\liealg{d}_4$, it follows that $(\lieh_1)^{\mu}$ is a Lie algebra of rank $2$ (\cref{prop-rank-of-fixed-point-set-of-automorphism}), and thus $\dim\lies = 2$. Because 
			$|\Omega^{+}| \leq \rank \lieg - \dim \lies$ by \cref{cor-estimate-on-rank}, this implies that $\Omega^{+}$ can contain at most $6$ elements. On the other 
			hand, by \cref{lem-maximal-torus} there exists a maximal torus $\liealg{u} \subseteq \liet_1$ of $\liek_1 \cap \lieg^{\mu}$, and if we let $\Gamma$ be the $\lieg$-roots 
			vanishing on $\liealg{u}$, then we may apply \cref{cor-estimate-on-rank} to the inner automorphisms $\sigma_1$, $\mu$ to conclude that 
			$|\Gamma^{+}| = \rank \lieg - \dim \liealg{u}$. Since $\dim\liealg{u} = \rank\, (\lieh_1)^{\mu} + \rank\, (\lieh_2)^{\mu}$, we see that $\dim\liealg{u} = 4$ and so 
			$|\Gamma^{+}| = 4$. Also, because $\lies \subseteq \liealg{u}$, we have $\Gamma^{+} \subseteq \Omega^{+}$. However, $\alpha$ cannot be contained in 
			$\Gamma^{+}$. Indeed, if $\alpha$ was to vanish on $\liealg{u}$, then we could apply \cref{lem-order-on-roots} and conclude that the smallest number $j$ with 
			$\mu^j(\alpha) = \alpha$ is even. As $\mu|_{\liet_1}$ is of order $3$, this would yield $j = 2$, a contradiction. Therefore, $\alpha \in \Omega^{+} - \Gamma^{+}$. Applying 
			\cref{lem-order-on-roots} again, we see that $(\sigma_2)^3(\alpha) = -\alpha$ and that $6$ is the smallest integer $j$ with $(\sigma_2)^j(\alpha) = \alpha$. But by 
			\cref{lem-order-on-roots} the roots $\alpha$, $\sigma_2(\alpha)$, and $(\sigma_2)^2(\alpha)$ then must be linearly independent, and since $\Gamma$ is a $\sigma_2$
			invariant subset, this implies that $\Omega^{+} - \Gamma^{+}$ contains at least $3$ elements. Hence $|\Omega^{+}| \geq 7$, giving the desired contradiction.
		\end{proofofclaim}
	\end{proof}
	
	\begin{lemma}\label{lem-roots-reflected-on-a-series}
		If $\lieg$ is of type $\liealg{a}_n$, $n \geq 1$, then $\sigma_2(\alpha) = -\alpha$ for all $\alpha \in \Omega$.
	\end{lemma}
	\begin{proof}
		Assume that there exists a root $\alpha \in \Omega$ with $\sigma_2(\alpha) \neq -\alpha$ for a contradiction. By \cref{lem-order-at-most-four-on-vanishing-roots}
		this is only possible if $\mu(\alpha) = -\alpha$, where $\mu = (\sigma_2)^2$, and if furthermore $\liealg{q} := [\liek_1, \liek_1]$ decomposes into the sum of two
		simple ideals $\lieh_1$, $\lieh_2$ which are isomorphically mapped onto each other by $\sigma_2$. We invoke the classification of involutions on simple Lie algebras
		(\cref{prop-classification-of-involutions}) and deduce
		that $\sigma_1$ is an inner automorphism, that $\liek_1 = Z\oplus\liealg{q}$ has a one-dimensional center $Z$, and that $\lieh_1$ and $\lieh_2$ are of type $\liealg{a}_k$ for 
		some number $k \geq 1$. In particular, $n = 2k + 1$.
		
		Now as $\Out(G)$ is cyclic of order $2$, $\mu$ must be an inner automorphism. Hence, if we denote by $\Gamma$ the $\lieg$-roots vanishing on the maximal 
		torus $\liealg{u} = \liet^{\mu}$ of $(\liek_1)^{\mu}$, then $\dim \liealg{u} = n - |\Gamma^{+}|$ by \cref{cor-estimate-on-rank}. We also know
		from \cref{lem-order-at-most-four-on-vanishing-roots} that $\sigma_1$, $\mu$ satisfy property $\star$, as $\Gamma \subseteq \Omega$. \Cref{thm-normal-form-automorphisms} then shows that
		$\mu|_{\liet}$ is an involution, and so it follows from \cref{cor-inner-automorphisms-normal-form-for-roots-vanishing-on-maximal-torus} that the integer
		\begin{align*}
			\sum_{\alpha\in\Gamma^{+}} \frac{2\langle\alpha,\beta\rangle}{\langle\alpha,\alpha\rangle}
		\end{align*}
		is even for all $\lieg$-roots $\beta$. Then the classification in \cite[example II.4.11]{thesis} applies and shows that $|\Gamma^{+}| = k + 1$. 
		Consequently, $\dim\liealg{u} = k$ and $\rank\,(\liek_1)^{\mu} = k$. If $\lieh_i$ was of type $\liealg{a}_k$ with $k = 2r$ even, then 
		\begin{align*}
			2r &= k = \rank\,(\liek_1)^{\mu} = 2\cdot \rank\,({\lieh_1})^{\mu} + 1,
		\end{align*} 
		as $\mu$ fixes $Z$ pointwise. But we cannot have $k = 2r - 1$ either. Indeed, we know that $\mu|_{\lieh_i}$ cannot be an inner automorphism, for otherwise $\mu$ would fix
		$\liet \cap \lieh_i$ pointwise by \cref{lem-maximal-torus} and therefore all of $\liet$. This means that $(\lieh_i)^{\mu}$ is of rank $r$, because $\lieh_i$ is of type $\liealg{a}_{2r-1}$ (see \cref{prop-rank-of-fixed-point-set-of-automorphism}), and
		\begin{align*}
			2r - 1 &= k = \rank\,(\liek_1)^{\mu} = 2\cdot \rank\,({\lieh_1})^{\mu} + 1 = 2r + 1,
		\end{align*}
		which again yields a contradiction.
	\end{proof}
	
	As is well-known, in an abstract (reduced) root system, the reflection $s_{\alpha}$ along a simple root $\alpha$ permutes all the positive roots except for $\alpha$.
	
	\begin{lemma}\label{lemma-weyl-group-element-resembling-root-reflection}
		Let $(V, \langle\cdot,\cdot\rangle, \Delta)$ be an abstract root system, $\Delta^{+}$ a choice of positive roots, and 
		$\Phi$ the corresponding set of simple roots. Fix $\alpha \in \Phi$. If a Weyl group element satisfies $w(\alpha) = -\gamma$ and 
		$w(\beta) \in \Delta^{+}$ for all $\beta \in \Phi - \{\alpha\}$ and some positive root $\gamma$, then necessarily $\gamma = \alpha$ and $w = s_{\alpha}$.
	\end{lemma}
	\begin{proof}
		Define the two open Weyl chambers $C_{+}$ and $C_{-}$ by
		\begin{align*}
			C_{\pm} &:= \bigl\{ v \in V \,|\, \langle \pm\alpha, v\rangle > 0\bigr\} \cap \bigcap_{\beta \in \Phi-\{\alpha\}} \bigl\{ v \in V \,|\, \langle \beta, v\rangle > 0\bigr\}. 
		\end{align*}
		Note that if $v \in C_{+}$ is arbitrary, then $\langle \beta, w^{-1}(v)\rangle > 0$ for all $\beta \in \Phi-\{\alpha\}$, since $w(\beta)$ is a positive root 
		by assumption. On the other hand, $w(\alpha) = -\gamma$ is a negative root, and so we see that $w^{-1}(C_{+}) = C_{-}$. But $s_{\alpha}(\beta)$
		is also a positive root for $\beta \in \Phi - \{\alpha\}$, because $\alpha$ and $\beta$ are simple (see \cite[lemma 10.2.B]{Humphreys-lie-algebras}), and 
		$s_{\alpha}(\alpha) = -\alpha$. Therefore $s_{\alpha}(C_{+}) = C_{-}$ holds as well. Since the Weyl group acts simply transitively on the set of open Weyl 
		chambers (cf.\,\cite[theorem 10.3]{Humphreys-lie-algebras}), we thus must have $w = s_{\alpha}$, and in particular $w(\alpha) = -\alpha$. 
	\end{proof}
	
	\begin{lemma}\label{lem-roots-reflected-on-d-series}
		If $\lieg$ is of type $\liealg{d}_n$, $n \geq 4$, then $\sigma_2(\alpha) = -\alpha$ for all $\alpha \in \Omega$.
	\end{lemma}
	\begin{proof}
		As in the proof of \cref{lem-roots-reflected-on-a-series}, we argue by contraposition. Thus, assume that there exists a root $\alpha \in \Omega$ with 
		$\sigma_2(\alpha) \neq -\alpha$ for a contradiction. By \cref{lem-order-at-most-four-on-vanishing-roots}
		this is only possible if $\mu(\alpha) = -\alpha$, where $\mu = (\sigma_2)^2$, and if $[\liek_1, \liek_1] = \lieh_1\oplus\lieh_2$ is the sum of two
		simple ideals $\lieh_1$, $\lieh_2$ which are isomorphically mapped onto each other by $\sigma_2$.  
		Let $\Gamma$ be the roots vanishing on the maximal torus $\liealg{u} = \liet^{\mu}$ of $\lieg^{\sigma_1} \cap \lieg^{\mu}$. 		
		
		\begin{claim}\label{lem-roots-reflected-on-d-series-claim-1}
			The automorphism $\sigma_1$ is inner, $n = 2r$, $\liek_1 = \lieh_1\oplus\lieh_2$, and $\lieh_i \cong \liealg{d}_r$.
		\end{claim}
		\begin{proofofclaim}
			From \cref{prop-classification-of-involutions} we see that either $\lieh_i \cong \liealg{b}_r$, $n = 2r + 1$, and $\sigma_1$ is an outer automorphism; or else
			$\liek_1$ is semisimple, $\lieh_i \cong \liealg{d}_r$, and $\sigma_1$ is an inner automorphism. In the first case $\mu|_{\lieh_i}$ must be an inner automorphism, because
			the outer automorphism group of Lie algebras of type $\liealg{b}_r$ is trivial (see \cite[theorem IX.5.4]{Helgason} and \cite[theorem X.3.29]{Helgason}). By \cref{lem-maximal-torus} this would 
			mean that $\mu$ is the identity on $\liet_1 \cap \lieh_i$, hence on $\liet_1$. Thus, as $\dim \liet = \dim \liet_1 + 1$, there could be at most
			one root in $\Gamma^{+}$. But this is impossible, because $\Gamma^{+}$ is $\sigma_1$-invariant and both $\alpha$ and $\sigma_1(\alpha)$ are linearly independent by 
			\cref{lem-order-on-roots}. 		
		\end{proofofclaim}
		
		\begin{claim}\label{lem-roots-reflected-on-d-series-claim-2}
			The automorphism $\mu$ is inner.
		\end{claim}
		\begin{proofofclaim}
			Suppose that $\mu$ is not inner for a contradiction. Since $\Out(G)$ is isomorphic to the group of permutations on $3$ elements if and only if $n = 4$ and
			to $\ZZ/2\ZZ$ else, and because $\mu = (\sigma_2)^2$, we see that $n = 4$ and that both $\sigma_2$ and $\mu$ represent elements of order $3$ in $\Out(G)$. In particular, 
			$\lieg^{\mu}$ is a Lie algebra of rank $2$ by \cref{prop-rank-of-fixed-point-set-of-automorphism}. On the other hand, we know that $\liek_1 = \lieh_1\oplus\lieh_2$ is a sum of
			two simple ideals which are permuted by $\sigma_2$ and that the isomorphism type of $\lieh_i$ equals $\liealg{a}_1\oplus\liealg{a}_1$. Therefore, $(\lieh_i)^{\mu}$ is either of 
			rank $1$ or $2$, and $(\liek_1)^{\mu}$ accordingly is of rank $2$ or $4$. Now as $(\liek_1)^{\mu} \subseteq \lieg^{\mu}$ and the latter is a Lie algebra of rank $2$,
			it follows that $(\liek_1)^{\mu}$ must be of rank $2$ as well. In particular, $\lies$ is also a maximal torus of $\lieg^{\mu}$. But then $\lies$ must
			contain regular elements for the $\lieg$-roots \cite[lemma X.5.3]{Helgason}, which means that no $\lieg$-root can vanish on $\lies$.
			This, however, contradicts our assumption $\Gamma \neq \emptyset$.				
		\end{proofofclaim}
		
		\begin{figure}
			\includegraphics{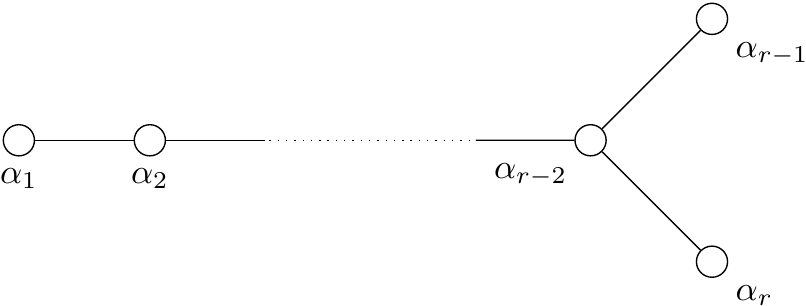}
			\caption{Dynkin diagram of the root system $\liealg{d}_r$.}
			\label{fig-dynkin-diagram-type-d}
		\end{figure}
			
		\begin{claim}\label{lem-roots-reflected-on-d-series-claim-3}
			Enumerating the simple roots as in \cref{fig-dynkin-diagram-type-d}, there exists a Weyl group element $w \in W(G, T)$ such that $w(\Gamma) = \{\pm\alpha_1, \pm\delta\}$.
			In particular, if $n > 4$, then $\alpha_1$ is fixed by the non-trivial Dynkin diagram automorphism.
		\end{claim}		
		\begin{proofofclaim}
			Because $\Gamma \subseteq \Omega$, we know from \cref{lem-order-at-most-four-on-vanishing-roots} that  $\sigma_1$ and $\mu$ satisfy property $\star$. 
			Hence, by \cref{thm-normal-form-automorphisms} and possibly after conjugating $\sigma_1$ and $\mu$ with a Weyl group element, we have 
			$\mu = \mu_1\circ \mu_2$, where $\mu_1$ is an inner automorphism of $G$, representing the Weyl group element $\prod_{\beta\in\Gamma^{+}} s_{\beta}$ on $\liet$, 
			and where $\mu_2$ is an automorphism permuting the positive $\lieg$-roots. Observe that $\mu$ represents a non-trivial element in $\Out(G)$ if and only
			if $\mu_2$ does so. Because $\mu$ is inner, it hence follows that $\mu_2$ is an inner automorphism permuting the positive $\lieg$-roots. Therefore, $\mu_2$ fixes $\liet$
			pointwise and $\mu|_{\liet} = \prod_{\beta\in\Gamma^{+}} s_{\beta}$. In particular, $\mu|_{\liet}$ is an involution.
			
			This means that $\mu^2$ fixes $\liet \cap \lieh_i$ pointwise. As $\liet \cap \lieh_i$ is 
			a maximal torus of 
			$\lieh_i$, $(\mu|_{\lieh_i})^2$ must be an inner automorphism, and therefore $\mu|_{\lieh_i}$ represents an element of order $2$ in $\Out(\lieh_i)$.
			Using that $\lieh_i$ is a Lie algebra of type $\liealg{a}_1\oplus\liealg{a}_1$ (if $r = 2$), $\liealg{a}_3$ ($r = 3$) or $\liealg{d}_r$ $(r \geq 4$),
			and that $\mu|_{\lieh_i}$ is not an inner automorphism (as otherwise $\mu$ would already fix $\liet$ pointwise), we hence conclude that $(\lieh_i)^{\mu}$ is a Lie algebra of rank $r - 1$. 
			Thus, $(\liek_1)^{\mu}$ is a Lie algebra of rank $2r - 2 = n - 2$ and $\dim \liealg{u} = n - 2$. Now it follows from \cref{cor-estimate-on-rank} that $\Gamma^{+}$ contains 
			exactly two elements.
			
			Finally, we observe that \cref{cor-inner-automorphisms-normal-form-for-roots-vanishing-on-maximal-torus} allows us to apply the classification of strongly orthogonal
			sets of roots in Lie algebras of type $\liealg{d}_n$ given in \cite[example II.4.12]{thesis}. Looking at this classification, we immediately find that, up to application of a Weyl 		
			group element, $\Gamma^{+}$ has the form claimed. Indeed, according to this classification $\Gamma^{+}$ either contains $n$ elements, which is impossible, or must be 
			a set of the form
			\begin{align*}
				\{\alpha_1, \delta_1, \alpha_3, \delta_3, \ldots, \alpha_{2t-1}, \delta_{2t-1}\}\text{, }\{\delta_1, \delta_3, \ldots, \delta_{n-3}, \gamma\}.
			\end{align*}
			Here, $\alpha_1, \ldots, \alpha_n$ are the simple roots, enumerated as in \cref{fig-dynkin-diagram-type-d}, $\delta_i$ is the highest root of the root subsystem of $(\Delta, \I\liet_{\RR})$
			determined by $\Delta \cap \operatorname{span}_{\ZZ}\{\alpha_i, \ldots, \alpha_n\}$, and $\gamma \in \{\alpha_{n-1}, \alpha_n\}$. Also, the integer
			$t$ satisfies $2t - 1 < n - 2$. Hence, either $\Gamma^{+}$ is of the first form and $t = 1$ or $n = 4$ and $\Gamma^{+}$ is of the second form. In both cases,
			$\Gamma^{+}$ is as claimed, because $\delta_1$ is the highest root. 
		\end{proofofclaim}		
		
		By the previous claim, we may assume that $\Gamma^{+} = \{\alpha_1, \delta\}$, where $\alpha_1, \ldots, \alpha_n$ are enumerated as in \cref{fig-dynkin-diagram-type-d} and
		$\delta$ is the highest root. Write 
		$\sigma_2 = c_g\circ\nu$ for some element $g \in N_G(T)$ and some automorphism of $G$ with $\nu(\Delta^{+}) = \Delta^{+}$. 
		Consider also the set $\Delta' \subseteq \Delta$ consisting of all roots which are perpendicular to $\alpha_1$ and $\delta$,
		and let $\liet' = \sum_{\beta\in\Delta'} \CC H_{\beta}$ and define $\lieg'$ to be the Lie subalgebra
		\begin{align*}
			\lieg' &= \liet'\oplus\bigoplus_{\beta\in\Delta'}\lieg_{\beta}
		\end{align*}
		of $\lieg$. This Lie algebra is semisimple, so the connected subgroup $G'$ of $G$ with Lie algebra $\lieg'$ is compact, and we denote by 
		$\sigma_i'\colon G'\rightarrow G'$ the restriction of $\sigma_i$ to $G'$.
		
		\begin{claim}\label{lem-roots-reflected-on-d-series-claim-4}
			$\nu$ fixes $\Gamma^{+}$ pointwise.
		\end{claim}
		\begin{proofofclaim}
			For $\delta$ this is immediate, because any 
			automorphism permuting the positive simple roots must fix the highest root. Also, if $n > 4$, then part of the statement of \cref{lem-roots-reflected-on-d-series-claim-3} is 
			that $\alpha_1$ is fixed by the non-trivial Dynkin diagram automorphism on the simple roots, whence $\nu$ fixes $\alpha_1$ in this case as well. Therefore, we may suppose 
			that $n = 4$. Arguing by contraposition, we shall also suppose that $\nu(\alpha_1) \neq \alpha_1$. 
			By \cref{thm-normal-form-automorphisms} there exists $w' \in W(G', T')$ such that $w'(\sigma_2')(w')^{-1} = c_{g'}\circ\nu'$ is a composition of an automorphism $\nu'$ permuting the 
			positive roots $(\Delta')^{+}|_{\liet'}$, and an inner automorphism $c_{g'}$ which on $\liet'$ is a product of root reflections that commute with $\nu'$. Because
			$w'$ extends to $G$ and fixes $\Gamma^{+}$ pointwise, there will be no loss of generality if we assume $w' = \id$ (notice that conjugating $\sigma_2$ by $w'$ changes $g$, but 
			not $\nu$). Moreover, since $\lieg$ is of type $\liealg{d}_4$, we have enumerated the simple roots $\alpha_1$, $\alpha_2$, $\alpha_3$, $\alpha_4$ of $\lieg$ in such a way 
			that $\alpha_2$ is the unique simple root not perpendicular to $\delta$. Then $(\Delta')^{+} = \{\alpha_3, \alpha_4\}$ and $\lieg'$ is a Lie 
			algebra of type $\liealg{a}_1\oplus\liealg{a}_1$, the two ideals being $\CC H_{\alpha_j}\oplus\lieg_{\alpha_j}\oplus\lieg_{-\alpha_j}$ for $j = 3, 4$. Hence,
			as $\dim\lies = 1$, we see that either $\sigma_2' = \Ad_{g'}$ is the product of exactly one root reflection or that else $\sigma_2' = \nu'$ and $\nu'$ interchanges
			the ideals of $\lieg'$. We discuss both cases separately and show that they both lead to a contradiction.
			
			Suppose first that $\sigma_2'$ interchanges the ideals of $\lieg'$. Because $\nu'$ permutes the simple roots $\{\alpha_3|_{\liet'}, \alpha_4|_{\liet'}\}$ of $\lieg'$, this means that 
			$\sigma_2$ interchanges $\alpha_3$ and $\alpha_4$. Since $\nu$ does not fix $\alpha_1$, we have $\nu(\alpha_1) = \alpha_i$ for $i \in \{3, 4\}$, and there is no loss of 
			generality if we assume that $i = 3$. Since $\nu$ and $\sigma_2$ represent the same element in $\Out(G)$ and $\mu = (\sigma_2)^2$ is an inner automorphism by 
			\cref{lem-roots-reflected-on-d-series-claim-2}, we have $\nu^2|_{\liet} = \id$,
			$\nu$ permutes two simple roots and fixes the other ones. Hence, $\nu(\alpha_3) = \alpha_1$, and 
			\begin{align*}
				\alpha_4 &= \sigma_2(\alpha_3) = \Ad_g(\alpha_1).
			\end{align*}
			Similarly, because $\nu(\alpha_4) = \alpha_4$, we have that $\Ad_g(\alpha_4) = \alpha_3$. Now note that $\Gamma$ consists of a single orbit of $\sigma_2$ by 
			\cref{lem-order-on-roots}. Therefore, either $\sigma_2(\alpha_1) = \delta$ and $\sigma_2(\delta) = -\alpha_1$ or $\sigma_2(\alpha_1) = -\delta$ and $\sigma_2(\delta) = \alpha_1$.
			The first case gives $\Ad_g(\alpha_3) = \delta$, so $\Ad_g$ maps $\alpha_1$, $\alpha_4$, and $\alpha_3$ onto positive roots. By \cref{lemma-weyl-group-element-resembling-root-reflection} $\Ad_g(\alpha_2)$ cannot be a negative root, so $\Ad_g(\alpha_2)$ must be positive as well. But then $g \in T$ and in particular $\Ad_g|_{\liet} = \id$,
			a contradiction. In the second case the element $\Ad_{g^{-1}}$ sends $\alpha_1$, $\alpha_4$, $\alpha_3$ onto positive roots, and we again obtain a contradiction.
			Therefore, $\sigma_2$ cannot interchange $\alpha_3$ and $\alpha_4$.
			
			 Now assume that $\sigma_2'|_{\liet'}$ is a root reflection. Since $\Delta' = \{\alpha_3, \alpha_4\}$, we may assume, renaming $\alpha_3$ and $\alpha_4$ if necessary, that
			 $\sigma_2(\alpha_3) = -\alpha_3$ and $\sigma_2(\alpha_4) = \alpha_4$. We have already seen that $\nu$ must be an automorphism of order $2$ on $\liet$, hence
			 either $\nu(\alpha_1) = \alpha_3$ or $\nu(\alpha_1) = \alpha_4$. In the first case, $\nu(\alpha_4) = \alpha_4$, and thus $\Ad_g(\alpha_4) = \alpha_4$. By
			 Chevalley's Lemma, \cite[proposition 2.72]{Knapp}, $\Ad_g$ then must be a product of roots $\beta$ with $\beta \in \{\alpha_1, \alpha_3, \delta\}$. Observe that all of these roots are 
			 mutually perpendicular and that on the other hand, since $\nu$ fixes $\delta$, we also have $\Ad_g(\delta) = \pm\alpha_1$. But no product of such roots $\beta$
			 maps $\delta$ onto $\pm\alpha_1$, and therefore this case cannot apply. If we assume that $\nu(\alpha_1) = \alpha_4$, then $\Ad_g(\alpha_3) = -\alpha_3$.
			 Again it follows from Chevalley's Lemma that $\Ad_g$ is a product of $s_{\alpha_3}$ and a product of roots $\beta$ with $\beta \in \{\alpha_1, \alpha_4, \delta\}$,
			 and again it follows from this description that $\delta$ cannot be mapped onto $\pm\alpha_1$ by $\Ad_g$.		 
		\end{proofofclaim}
		
		The contention of the lemma now follows thus. Since $\nu$ leaves $\liet'$ invariant by \cref{lem-roots-reflected-on-d-series-claim-4}, also $\Ad_g$ must leave $\liet'$ invariant. Hence, if we write 
		$\sigma_2' = c_{g'}\circ\nu'$, where again $\nu'$ is an automorphism permuting the positive roots $(\Delta')^{+}|_{\liet'}$ on $\lieg'$ and $g' \in N_{G'}(T')$, then
		\begin{align*}
			(\Ad_g)|_{\liet'}\circ \nu|_{\liet'} &= \sigma_2' = c_{g'}\circ \nu'. 
		\end{align*} 
		By construction, $\nu|_{\liet'}$ and $\nu'$ preserve the chosen notion of positivity on the $\lieg'$-roots, and therefore $\Ad_{g^{-1}g'}|_{\liet'}$ permutes the simple 
		roots $\alpha_3|_{\liet'}$, \ldots, $\alpha_n|_{\liet'}$ of $\lieg'$. That is, $w = \Ad_{g^{-1}g'}|_{\liet}$ is a Weyl group element permuting $\{\alpha_3, \ldots, \alpha_n\}$.
		On the other hand, since $\Ad_{g'}$ fixes $\Gamma^{+}$ pointwise, we see that $(\sigma_2)^{-1}$ coincides with $w$ on $\Gamma^{+}$, hence either
		$w(\alpha_1) = \delta$ and $w(\delta) = -\alpha_1$; or $w(\alpha_1) = -\delta$ and $w(\delta) = \alpha_1$. In the first case
		$w$ sends all simple roots different from $\alpha_2$ to a positive root, so if also $w(\alpha_2)$ is a positive root, then $w = \id$, and if $w(\alpha_2)$ is a negative
		root, then $w = s_{\alpha_2}$ by \cref{lemma-weyl-group-element-resembling-root-reflection}. Therefore, the first case cannot apply. Assume then that the
		second case holds. Recall that $\lieg$ is of type $\liealg{d}_n$, so the highest root $\delta$ is $\delta = \alpha_1 + 2\alpha_2 + \gamma$ with 
		$\gamma \in \operatorname{span}_{\ZZ}\{\alpha_3, \ldots, \alpha_n\}$, see e.\,g.\,\cite[table 2, Section\,III.12.2]{Humphreys-lie-algebras}. We have
		\begin{align*}
			\alpha_1 &= w(\delta) = -\delta + 2w(\alpha_2) + w(\gamma),
		\end{align*}
		and we have seen that $w(\gamma)$ still is contained in the span of $\alpha_3, \ldots, \alpha_n$. It follows that the coefficient $m_1$ in 
		$w(\alpha_2) = \sum_j m_j\alpha_j$ must be equal to $1$, and hence $w(\alpha_2)$ is a positive root. But this is impossible as well,
		because now by \cref{lemma-weyl-group-element-resembling-root-reflection} we must have $w = s_{\alpha_1}$.
	\end{proof}
	
	\begin{proof}[Proof of \cref{thm-reflection}]
		Let $\liealg{q} = [\liek_1, \liek_1]$ be the semisimple part of $\liek_1$. According to \cref{lem-order-at-most-four-on-vanishing-roots}, we may assume that
		$\liealg{q} = \lieh_1\oplus\lieh_2$ is the sum of two simple ideals which are interchanged by $\sigma_2$. We set $\mu = (\sigma_2)^2$.
		Suppose that $\mu|_{\lieh_i}$ is an inner automorphism. Then $\mu$ fixes $\liet_1 \cap \lieh_i$ pointwise by \cref{lem-maximal-torus}, so all of $\liet_1$ is fixed pointwise
		by $\mu$, because the center of $\liek_1$ is at most one-dimensional. For each root $\alpha \in \Omega$ the vector $H_{\alpha}$ is fixed by $\sigma_1$ by \cref{prop-strongly-orthogonal}, so that $H_{\alpha} \in \liet_1$, 
		and hence we have $\mu(\alpha) = \alpha$ and $\sigma_2(\alpha) = -\alpha$ for all $\alpha \in \Omega$ by \cref{prop-some-power-of-vanishing-root-is-reflected}. 
		Thus, we may assume that $\mu|_{\lieh_i}$ is not an inner automorphism. This is only possible if $\lieh_i$ is of type $\liealg{a}_n$, $\liealg{d}_n$ or $\liealg{e}_6$.
		According to the classification of fixed point sets of involutions (\cref{prop-classification-of-involutions}) the Lie algebra $\lieg$ then must be of type $\liealg{a}_n$ or $\liealg{d}_n$, whence the statement follows
		from \cref{lem-roots-reflected-on-a-series} and \cref{lem-roots-reflected-on-d-series}.
	\end{proof}
	
	\section{Isotropy formality in case that $\sigma_1$ and $\sigma_2$ are of the same type}\label{sec-case-same-type}
	
	Let $G$ be a compact, connected Lie group and $\sigma_1$, $\sigma_2$ automorphisms defining an $\ZZ_2\times\ZZ_k$-action on $G$. We adopt the
	notation of the previous sections. Thus, $\sigma_1$ is an involution, $K_1$ is the connected component of its fixed point set, $K = (G^{\sigma_1} \cap G^{\sigma_2})_0$,
	and $S \subseteq K$ is a maximal torus. Our preferred choices of maximal tori in $K_1$ and $G$ are $T_1 = Z_{K_1}(S)$ and $T = Z_G(T_1)$, respectively. Provided
	that $G$ is simple we have established, up to conjugation, a normal form for $\sigma_2$ in \cref{thm-reflection,thm-normal-form-automorphisms}, and using this normal form we will now 
	show that the pair $(G, K)$ is isotropy formal. This will prove \cref{thm-main} in case that $G$ is simple, and then also for arbitrary Lie groups by \cref{thm-reduction-principle}.
	We divide the proof of \cref{thm-main} for simple $G$ into three cases: that $\sigma_1$ and $\sigma_2$ are both inner or both outer, that $\sigma_1$ is outer and
	$\sigma_2$ is inner, and finally the case in which $\sigma_1$ is inner but $\sigma_2$ is outer. While the proof of the last case will again make considerable use of the
	classification of fixed point sets of involutions (\cref{prop-classification-of-involutions}), the first case does less so and is motivated by the following observation.
	
	According to \cref{thm-normal-form-automorphisms}, $\sigma_2$ is a composition of certain root reflections $s_{\alpha}$ and an automorphism $\nu$ of $G$ that is induced by an automorphism of
	the Dynkin diagram. Furthermore, each of the factors in this composition commute with each other. Therefore, if $\nu|_{\liet}$ is of order at most $2$, which for example is always
	the case if $\lieg$ is simple and its isomorphism type is different from $\liealg{d}_4$, then $\sigma_2|_{\liet}$ too has order at most $2$. A natural question thus is whether 
	isotropy formality of $(G, K)$ can be deduced from equivariant formality of isotropy actions of $\ZZ_2\times\ZZ_2$-symmetric spaces. Note, that even though
	$s_{\alpha}\colon\liet\rightarrow\liet$ is an involution for every root $\alpha$, it may not be possible to represent $s_{\alpha} \in W(G, T)$ by some element in $g \in N_G(T)$ 
	which is an involution on all of $G$. This is related to the fact that the exact sequence
		$0 \rightarrow T \rightarrow N_G(T) \rightarrow W(G, T) \rightarrow 0$
	does not necessarily split \cite{AdamsHe-weyl-group}. Therefore, there is no reason to believe that $\sigma_2$ generally might be represented by an involution. If, however,
	$\sigma_1$ and $\sigma_2$ represent the same automorphism in $\Out(G)$, then isotropy formality of $(G, K)$ can be deduced from the known results about
	$\ZZ_2\times\ZZ_2$-symmetric spaces.
	
	\begin{theorem}\label{cor-eq-formality-in-good-cases}
		Suppose that $\sigma_1\circ(\sigma_2)^{-1}$ is inner and that $\sigma_1$, $\sigma_2$ satisfy property $\star$. Then $(G, K)$ is isotropy formal.
	\end{theorem}
	
	\begin{lemma}\label{lem-good-notion-of-positivity}
		Under the assumptions of \cref{cor-eq-formality-in-good-cases} there exists a notion of positivity $\Delta^{+}$ such that $\sigma_1 = \nu$ 
		and $\sigma_2 = c_g\circ\nu$ where $\Ad_g|_{\liet} = \prod_{\alpha\in\Omega^{+}} s_{\alpha}$ and where $\nu$ is an automorphism of $G$ 
		which fixes $\Omega$ pointwise and permutes $\Delta^{+}$.
	\end{lemma}
	\begin{proof}
		Because $\liet = Z_{\lieg}(\liet_1)$, there is a notion of positivity $\Delta^{+}$ on $\Delta$ which is preserved by $\sigma_1$. We now apply \cref{thm-normal-form-automorphisms} to choose an 
		element $w \in W(G, T)$ such that 
		\begin{align*}
			w\sigma_2w^{-1}|_{\liet} &= \prod_{\alpha\in\tilde{\Omega}^{+}} s_{\alpha}\circ\nu|_{\liet}
		\end{align*}
		for $\tilde{\Omega} = w(\Omega)$ and some automorphism $\nu\colon G\rightarrow G$ which permutes $\Delta^{+}$ and fixes each element in 
		$\tilde{\Omega}^{+}$. Note that $w^{-1}\nu w = c_k\circ \nu$ for some $k \in N_G(T)$. Therefore, we find that $\sigma_2 = c_g\circ \nu$ for some element 
		$g \in N_G(T)$. The assumption that $\sigma_1\circ (\sigma_2)^{-1}$ is an inner automorphism means that $\sigma_1\circ\nu^{-1}$ is inner. But $\sigma_1\circ\nu^{-1}$ 
		preserves $\Delta^{+}$, and therefore $\sigma_1\circ \nu^{-1}$ is an inner automorphism fixing $T$ pointwise. So, replacing $\nu$ by $c_t\circ \sigma_1$ for some 
		$t \in T$ if necessary, we may assume that $\sigma_1 = \nu$ and that $\sigma_2 = c_g\circ \nu$. In particular, the above description of $w\sigma_2w^{-1}|_{\liet}$ shows
		that $\sigma_2|_{\liet}$ is an involution. Let us now prove that $\Ad_g|_{\liet}$ is a product of root reflections which commute with $\nu|_{\liet}$. We set 
		$r := \prod_{\alpha\in\Omega^{+}} s_{\alpha}$ and $f := w^{-1}\circ \nu|_{\liet}\circ w\circ (\nu|_{\liet})^{-1}$. This gives  
		\begin{align*}
			\sigma_2|_{\liet} &= \prod_{\alpha\in\Omega^{+}} s_{\alpha}\circ w^{-1}\circ \nu\circ w = r \circ f\circ \nu.
		\end{align*}
		
		As a first step, we show that $f$ leaves the space $V = \bigcap_{\alpha\in\Omega^{+}} \ker \alpha$ invariant. Indeed, if $\beta \in \Omega$ is arbitrary, then since 
		$\sigma_1 = \nu$, we have $\nu(\beta) = \beta$ by \cref{prop-strongly-orthogonal}. Because $\sigma_1$ and $\sigma_2$ satisfy property $\star$ by assumption, we also have 
		$-\beta = \sigma_2(\beta)$, and this shows that $f(\beta) = \beta$, because $r(\beta) = -\beta$ and $r$ is a bijection. As 
		$\liet = V\oplus \bigoplus_{\alpha\in\Omega^{+}} \CC H_{\alpha}$ by \cref{prop-strongly-orthogonal}, and since $f$ is an isometry, we conclude that $f(V) = V$.
		Next, we show that $f\circ \nu|_{\liet} = \nu|_{\liet}\circ f^{-1}$. In fact, since $f(V) = V$ and $r|_V = \id_V$, we see that $f$ and $r$ commute. On the other hand, because 
		$\sigma_2|_{\liet}$, $r$, and $\nu$ are involutions, and because $r$ commutes with $\nu = \sigma_1$, we also have
		\begin{align*}
			\id_{\liet} &= (r\circ f\circ \nu|_{\liet})\circ (r\circ f\circ \nu|_{\liet}) =   f\circ \nu|_{\liet} \circ f\circ (\nu|_{\liet})^{-1},
		\end{align*} 
		and this shows that $f\circ \nu|_{\liet} = \nu|_{\liet}\circ f^{-1}$. 
		Now decompose $V = V^{+}\oplus V^{-}$ into the $1$- and $(-1)$-eigenspaces of the involution $\nu|_{\liet}$. We claim that $f$ also preserves $V^{+}$ and $V^{-}$.
		To see this, recall that $\sigma_2$ commutes with $\sigma_1$, hence $\sigma_2$ preserves $V^{+}$ and $V^{-}$. But $r|_V = \id$ and $\nu$ preserves $V^{+}$ and
		$V^{-}$, hence so must do $f$.
		This now allows us to prove that $f$ must be an involution. In fact, if $v \in V$ is an eigenvector for $\nu$ to the eigenvalue $\lambda$, then
		\begin{align*}
			\lambda\cdot f(v) &= f(\nu(v)) = \nu(f^{-1}(v)) = \lambda\cdot f^{-1}(v),
		\end{align*}		 
		and this gives $f^2|_V = \id_V$. As we have already seen that $f$ fixes each vector $H_{\beta}$ with $\beta \in \Omega$, we see that $f$ is an involution.
		In particular, $f\circ \nu|_{\liet} = \nu|_{\liet} \circ f$.
		
		But note that $f \in W(G, T)$, because $\nu$ leaves $T$ invariant, and that a Weyl group element can only be an involution if it is a product of commuting
		root reflections, cf.\,\cite[problem II.12.13]{Knapp}. Therefore, $f = \prod_{\alpha \in \Phi} s_{\alpha}$ for some set of roots $\Phi$. However,
		if $\alpha \in \Phi$, then for any $X \in \lies$ we have
		\begin{align*}
			\langle H_{\alpha}, X\rangle &= \langle \sigma_2(H_{\alpha}), X\rangle = \langle (r\circ f\circ \nu)(H_{\alpha}), X\rangle = -\langle H_{\alpha}, X\rangle,
		\end{align*}
		because $r$ and $\sigma_1$ are isometries of $\langle\cdot,\cdot\rangle$ which fix $\lies$ and commute with $f$. This shows that
		$\alpha$ restricts to zero on $\liealg{s}$, and hence $\Phi \subseteq \Omega$, which by \cref{prop-some-power-of-vanishing-root-is-reflected} is only possible if $\Phi = \emptyset$.
	\end{proof}
	
	\begin{proof}[Proof of \cref{cor-eq-formality-in-good-cases}]
		We choose a notion of positivity $\Delta^{+}$ as in \cref{lem-good-notion-of-positivity}. Thus, $\sigma_1 = \nu$ and $\sigma_2 = c_g\circ \nu$ 
		for an automorphism $\nu$ which permutes $\Delta^{+}$ and fixes each root in $\Omega$. Moreover, $\Ad_g|_{\liet}$ is the product $\prod_{\alpha\in\Omega^{+}} s_{\alpha}$ 
		of root reflections $s_{\alpha}$ with $\alpha \in \Omega^{+}$. 
		
		Because $\nu = \sigma_1$ and $\sigma_2$ commute, we have $c_g\circ \nu = \nu\circ c_g$, so $\nu$ and $c_g$ commute. As is well-known (see 
		e.\,g.\,\cite[theorem IV.4.54]{Knapp} or \cref{prop-simple-root-reflections} below), each root reflection $s_{\alpha}$ can be represented by 
		$\exp(E_{\alpha} + \overline{E}_{\alpha})$ for some element $E_{\alpha} \in \lieg_{\alpha}$. We now choose one such element $x_{\alpha} \in N_G(T)$ for each 
		$\alpha \in \Omega^{+}$. Since $\lieg_{\alpha}$ is contained in the $(-1)$-eigenspace of $\nu$ for all $\alpha \in \Omega$ (\cref{prop-strongly-orthogonal}), we then have 
		$\nu(x_{\alpha}) = (x_{\alpha})^{-1}$. Let then $x = \prod_{\alpha\in\Omega^{+}} x_{\alpha}$, which is a well-defined expression, because the factors of this product 
		commute with each other by \cref{prop-strongly-orthogonal}. As already observerd earlier, $g$ and $x$ represent the same 
		Weyl group element. Thus there exists some $t \in T$ with $g = xt$, and because $\liet$ decomposes into the $1$- and $(-1)$-eigenspaces of the involution $\nu$, we can 
		accordingly decompose $t$ as $t = t_{+}t_{-}$ with $\nu(t_{+}) = t_{+}$ and $\nu(t_{-}) = (t_{-})^{-1}$. Consider then the element $h := g\cdot (t_{+})^{-1}$. Since 
		$\nu(t_{+}) = t_{+}$, we still have $c_h\circ\nu = \nu\circ c_h$, that is to say, $\nu$ and $\tau_2 := c_h\circ\nu$ commute. In addition, we have 
		\begin{align*}
			h\cdot \nu(h) &= (xt_{-})\cdot \nu(xt_{-}) \\
			&= (xt_{-})\cdot (xt_{-})^{-1},
			\intertext{and this implies that $\tau_2$ is an involution. Indeed, $\nu$ is an involution, whence}
			(\tau_2)^2 &= c_{h\nu(h)}\circ \nu^2 \\
			&= \id.
		\end{align*}
		We claim that $\lies$ is a maximal torus of $\lieg^{\nu} \cap \lieg^{\tau_2}$. To see this, suppose that $Y \in Z_{\lieg}(\lies)$ is fixed by $\nu$ and $\tau_2$. Because each 
		$\lieg_{\alpha}$ with $\alpha \in\Omega$ is contained in the $(-1)$-eigenspace of $\nu$, we must have $Y \in \liet_1$. Moreover, 
		$\sigma_2|_{\liet} = \tau_2|_{\liet}$, and hence $Y \in (\liet_1)^{\sigma_2}$. Since the latter space equals $\lies$ by construction, we see that $\lies$ is indeed a 
		maximal torus of $H := (G^{\nu} \cap G^{\tau_2})_0$. Then the pair $(G, H)$ is isotropy formal, because $\nu$ and $\tau_2$ are involutions 
		\cite{AmannKollross} (or \cite{thesis}), and hence also $(G, K)$ is isotropy formal by \cref{prop-formality-depends-on-torus}, because $S$ is maximal torus for both $K$ and $H$.
	\end{proof}
	
	\begin{corollary}
		If $G$ is simple and $\sigma_1$ and $\sigma_2$ are both inner automorphisms or both outer automorphisms, then the isotropy action is equivariantly formal.
	\end{corollary}
	\begin{proof}
		If $G$ is simple, then $\sigma_1$ and $\sigma_2$ satisfy property $\star$ by \cref{thm-reflection}. Thus, if $\sigma_1$ and $\sigma_2$ are inner automorphisms
		or if $\lieg$ is a Lie algebra of isomorphism type different from $\liealg{d}_4$, then $(G, K)$ is isotropy formal by \cref{cor-eq-formality-in-good-cases}. Suppose
		now that $\lieg$ is of type $\liealg{d}_4$ and that $\sigma_1$ and $\sigma_2$ are outer automorphisms. By the classification of fixed point sets of involutions
		(\cref{prop-classification-of-involutions}), $\liek_1$ then must be a sum of (at most) two simple ideals of type $\liealg{b}_p$, $\liealg{b}_{3-p}$. These ideals can never be 
		isomorphic to each other, so must be preserved by $\sigma_2$, and since Lie algebras of type $\liealg{b}_k$ only admit inner automorphisms, we see that $\liealg{s}$ 
		must be a maximal torus of $\liek_1$. It follows that $(G, K)$ is isotropy formal, because of \cite{Goertsches-symmetric} or \cite[theorem 1.1]{Goertsches-automorphisms}. 
	\end{proof}
	
	Before closing this section, let us collect some results about the order of a Weyl group element and the order of its representatives, as in the later sections we will have to be able to 
	decide when a Weyl group element lifts to an element of the same order in $\mathrm{Int}(G)$. Let then $\alpha$ be a root and recall that $\langle\cdot,\cdot\rangle$ is 
	negative-definite. Thus, we may choose $E_{\alpha} \in \lieg_{\alpha}$ such that 
	\begin{align*}
		\langle E_{\alpha}, \overline{E_{\alpha}}\rangle &= -\frac{\pi^2}{2\langle\alpha,\alpha\rangle},
	\end{align*}
	As is well-known, if we let $g = \exp(E_{\alpha} + \overline{E_{\alpha}})$, then we have $\Ad_g = s_{\alpha}$ on $\liet$, see for example \cite[theorem IV.4.54]{Knapp}. 
	A straightforward computation using that the adjoint-representation is diagonalizable then determines $(\Ad_g)^2$ on each root space:

	\begin{proposition}\label{prop-simple-root-reflections}
		Set $Z = E_{\alpha} + \overline{E_{\alpha}}$ and $g = \exp(Z)$. Then we have
		\begin{enumerate}
			\item 
			$(\Ad_g)^2(E_{\alpha}) = E_{\alpha}$.
			
			\item 
			If $\beta$ is a root such that neither $\beta - \alpha$ nor $\beta + 3\alpha$ is a root, then 
			\begin{align*}
				(\Ad_g)^2|_{\lieg_{\beta}} &= e^{\I\pi\frac{2\langle\alpha,\beta\rangle}{\langle\alpha,\alpha\rangle}}\id.
			\end{align*}
		\end{enumerate}        
	\end{proposition}
	\begin{proof}
	Set $U = [\overline{E}_{\alpha}, E_{\alpha}]$, so that $U = \pi^2/(2\langle\alpha,\alpha\rangle) H_{\alpha}$. We note that $\ad_Z$ is a skew-hermitian endomorphism with respect to the Hermitian inner product $(X, Y) \mapsto \langle X, \overline{Y}\rangle$ on $\lieg$. Therefore, $\ad_Z$ is diagonalizable with purely imaginary eigenvalues, and we will compute these eigenvalues on the $\ad_Z$-invariant (and hence $\Ad_g$-invariant) subspaces 
	\begin{align*}
	 	V &= \lieg_{-\alpha}\oplus \CC U\oplus \lieg_{\alpha}\text{ and }W = \bigoplus_{n\geq 0} \lieg_{\beta+n\alpha}.
	\end{align*}
	In fact, in terms of the basis $E_{-\alpha}$, $U$, $E_{\alpha}$ of $V$ the matrix of $\ad_Z|_V$ is given by
	\begin{align*}
		\begin{pmatrix}
			0   & \alpha(U)     & 0 \\
			-1  & 0             & 1 \\
			0   & -\alpha(U)    & 0
		\end{pmatrix},
	\end{align*}
	and since $\alpha(U) = \pi^2/2$, we have $p(t) = (-t)\cdot (t+\I\pi)\cdot (t-\I\pi)$ for the characteristic polynomial $p(t)$ of $\ad_Z|_V$. Thus, $\ad_Z$ is diagonalizable with
	 eigenvalues $0$, $\I\pi$ and $-\I\pi$. Then using the series expansion of $\Ad_g$ it follows that $\Ad_g$ is diagonalizable with eigenvalues $e^0 = 1$ and $e^{\pm\I\pi} = -1$. 
	 But then necessarily $(\Ad_g)^2|_V = \id_V$.
	
	To compute the eigenvalues of $\ad_Z|_W$ we consider the basis $((\ad_Z)^k(E_{\beta}))_\ell$ of $W$, where $\ell = 0, \ldots, k$ and $k$ is the maximal integer such that 
	$\beta + k\alpha$ is a root. This is indeed a basis of $W$, because the $\alpha$-string through $\beta$ has no gaps. In particular, $k \in \{0, 1, 2\}$ by assumption, and since 
	we treated the case $\beta = \alpha$, we may assume $\beta \neq \alpha$. Then if $k = 0$, $\ad_Z = 0$ on $W$ and hence $\Ad_g = \id$. Thus, we may assume that either 
	$k = 1$ or $k = 2$, and then the matrix of $\ad_Z|_W$ is given, respectively, by 
	\begin{align*}
		\begin{pmatrix}
			0   & \beta(U)    \\
			1   & 0
		\end{pmatrix}
	\text{ or }
		\begin{pmatrix}
			0   & \beta(U)  & 0 \\
			1   & 0         & (\alpha+2\beta)(U)    \\
			0   & 1         & 0
		\end{pmatrix}.
	\end{align*}
	As noted before, $\alpha(U) = \pi^2/2$. Moreover, by assumption, $\beta - \alpha$ is not a root, so the $\alpha$-string through $\beta$ contains 
	only roots of the form $\beta + j\alpha$, and we have
	\begin{align*}
	 \beta(U) &= \frac{\pi^2}{2\langle\alpha,\alpha\rangle}\cdot\langle\beta,\alpha\rangle
	 = \frac{\pi^2}{4}\cdot \frac{2\langle\alpha,\beta\rangle}{\langle\alpha,\alpha\rangle}
	 = - k\frac{\pi^2}{4},
	\end{align*}
	cf. \cite[proposition II.2.48]{Knapp} or \cite[proposition II.8.4]{Humphreys-lie-algebras}. Thus, in case that $k = 1$, the characteristic polynomial $q(t)$ of $\ad_Z|_W$ is given by $q(t) = (t+\I\frac{\pi}{2})(t-\I\frac{\pi}{2})$. Arguing as before, we see that $(\Ad_g)^2$ is diagonalizable with eigenvalues $e^{\pm\I\pi} = -1$, so $(\Ad_g)^2 = -\id$ on $W$, as claimed. If $k = 2$, then 
	\begin{align*}
		q(t) &= (-t)\cdot(t^2 - (\alpha+3\beta)(U)) = (-t)\cdot (t+\I\pi)\cdot(t-\I\pi), 
	\end{align*}
	and $(\Ad_g)^2 = \id$ on $W$, which again agrees with our claim.
     \end{proof}     

	\begin{corollary}\label{cor-simple-reflections-on-weight-spaces}
		Let $\alpha$ be a root. If $\lieg$ is simple, but not of type $\lieg_2$, we have, for all roots $\beta$:
		\begin{align*}
		(\Ad_g)^2|_{\lieg_{\beta}} = e^{\I\pi\frac{2\langle\alpha,\beta\rangle}{\langle\alpha,\alpha\rangle}}\cdot \id.
		\end{align*}
	\end{corollary}
	\begin{proof}
		We first note that if the statement holds for a root $\alpha$ and if $w \in N_G(T)$ is arbitrary, then the statement also is true for the root 
		$\rho = \Ad_w(\alpha) = \alpha\circ(\Ad_w)^{-1}$. Indeed, $\Ad_w$ is a real endomorphism and an isometry of $\langle\cdot,\cdot\rangle$, whence $k = \exp(\Ad_w(Z))$ 
		represents the simple reflection defined by $s_{\rho}$. Moreover, for any root $\beta$ we have 
		\begin{align*}
			(\Ad_k)^2|_{\lieg_\beta} &= \Ad_w\circ (\Ad_g)^2\circ (\Ad_w)^{-1}|_{\lieg_{\beta}}
			= e^{\I\pi\frac{2\langle\alpha,(\Ad_w)^{-1}\beta\rangle}{\langle\alpha,\alpha\rangle}}\cdot\id
			= e^{\I\pi\frac{2\langle\rho,\beta\rangle}{\langle\rho,\rho\rangle}}\cdot\id.
		\end{align*}
		Because $\lieg$ is simple, all elements of the same length are contained in the same Weyl group orbit and there are at most two different root lengths 
		\cite[lemma 10.4.C]{Humphreys-lie-algebras}, so we only need to verify the assertion of the corollary for the highest long and the highest short root. Furthermore, after having chosen an 
		arbitrary notion of positivity, we only need to check the claimed identity for positive roots, because $\lieg_{\beta}$ is conjugate to $\lieg_{-\beta}$ and because $\Ad_g$ 
		commutes with conjugation, being induced by a real endomorphism on $\lieg_{\RR}$. Thus, suppose that $\beta$ is a positive root and that $\alpha = \delta$ is the highest 
		long root. We may assume $\beta \neq \delta$. Then $\beta  + \delta$ is not a root, for it would be higher than $\delta$. Moreover, $\delta - \beta$ is either a positive root or 
		not a root, and hence $2\delta - \beta$ cannot be a root, because it again would be higher than $\delta$. Thus, we see that the assumptions of 
		\cref{prop-simple-root-reflections} are met for the root $-\beta$, and hence 
		\begin{align*}
			(\Ad_g)^2|_{\lieg_{-\beta}} = e^{-\I\pi\frac{2\langle\delta,\beta\rangle}{\langle\delta,\delta\rangle}}\cdot \id.
		\end{align*}
		Using again that $\Ad_g$ is a real endomorphism, the assertion follows in this case.
		
		Now assume that there are two root lengths and that $\alpha = \gamma$ is the highest short root. Let $\Pi$ be the set of simple roots. Then $\lieg$ is of type $\liealg{c}_r$, 
		$\liealg{b}_r$ or $\liealg{f}_4$, and from the classification of Dynkin diagrams and their highest short roots (\cite[table 2, section\,III.12.2]{Humphreys-lie-algebras}) we know that there is at 
		most one root $\beta \in \Pi$ with $\langle\gamma,\beta\rangle \neq 0$. Moreover, since $\lieg$ is not of type $\lieg_2$, we also know that for this root $\beta$ the 
		$\gamma$-string through $\beta$ is of the form $\beta - j\gamma$ with $0 \leq j \leq 2$; in particular, $\beta + \gamma$ is not a root. Hence, as before we conclude that 
		$(\Ad_g)^2|_{\lieg_{\beta}}$ is as claimed. Since the root spaces of simple roots form a generating set and $\langle\gamma,\cdot\rangle$ is linear, the assertion holds for all 
		roots $\beta$.
	\end{proof}		
	
	\section{The case that $\sigma_1$ is outer and $\sigma_2$ is inner.}\label{sec-involution-outer-other-inner}
	
	We continue to use the notation introduced in the previous section and turn to the case that $\sigma_1$ is an outer involution and that $\sigma_2$ is inner.
	Consider then the automorphism $\tau_2 := \sigma_1\circ \sigma_2$. It is immediate that $\sigma_1$ and $\tau_2$ commute, because $\sigma_1$ and
	$\sigma_2$ commute. Write $H := (G^{\sigma_1} \cap G^{\tau_2})_0$. By assumption $\sigma_1\circ (\tau_2)^{-1} = (\sigma_2)^{-1}$ is an inner automorphism,
	and therefore the pair $(G, H)$ is isotropy formal by \cref{cor-eq-formality-in-good-cases}. Moreover, it follows exactly as in the proof of \cref{cor-eq-formality-in-good-cases}
	that $S$ must be a maximal torus for $H$. Indeed, if $Y \in Z_{\lieg}(\lies)$ is fixed by $\sigma_1$, then $Y \in \liet_1$, because $\sigma_1$
	restricts to $-\id$ on each root space $\lieg_{\alpha}$ with $\alpha \in \Omega$, and if in addition $Y$ is fixed by $\tau_2$ then $Y \in \lies$, because
	$\lies = (\liet_1)^{\sigma_2}$ and $\tau_2|_{\liet_1} = \sigma_2|_{\liet_1}$. By \cref{prop-formality-depends-on-torus}, also $(G, K)$ is isotropy formal. We have shown:
	
	\begin{theorem}\label{thm-formality-if-involution-is-outer-other-is-inner}
		If $\sigma_1$ is an outer automorphism and $\sigma_2$ is inner, then $(G, K)$ is isotropy formal.
	\end{theorem}
	
	\section{The case that $\sigma_1$ is inner and $\sigma_2$ is outer}\label{sec-involution-inner-other-outer}
	
	In this section we assume that $G$ is a simple, compact, connected Lie group and that $\sigma_1$ is an involute inner automorphism while $\sigma_2$ is a finite-order outer
	automorphism commuting with $\sigma_1$. We have no general argument to prove that the pair $(G, K)$ is isotropy formal in this case. However,
	there are only three series of simple Lie algebras admitting non-inner automorphisms, and therefore $\lieg$ must either be of
	isomorphism type $\liealg{e}_6$, $\liealg{a}_n$, or $\liealg{d}_r$. We will verify directly that in each of these cases the isotropy action of $K$ on $G/K$ is equivariantly formal. 
	Combined with the results of the previous two sections this will finish the proof that $(G, K)$ is isotropy formal whenever $G$ is simple. As already remarked earlier, 
	\cref{thm-reduction-principle} then shows that \cref{thm-main} also holds for arbitrary Lie groups $G$. 
    
	\subsection*{Lie groups of type $\liealg{e}_6$} Suppose first that $\lieg$ is of type $\liealg{e}_6$. According to the classification of isomorphism types of fixed point sets of involutive automorphisms 
	on simple Lie groups (\cref{prop-classification-of-involutions}), the fixed point set $\liek_1$ of $\sigma_1$ then is a sum of Lie algebras $\liek_1 = \liealg{i}_1\oplus\liealg{i}_2$ 
	where the pair $(\liealg{i}_1, \liealg{i}_2)$ either is of types $(\liealg{a}_1, \liealg{a}_5)$ or $(\CC, \liealg{d}_5)$. In particular, both $\liealg{i}_1$ and $\liealg{i}_2$ are invariant 
	under $\sigma_2$. Using \cref{prop-rank-of-fixed-point-set-of-automorphism} for each of the ideals $\liealg{i}_1$ and $\liealg{i}_2$, we conclude that the fixed point set of 
	$\sigma_2$ on $\liek_1$ must be at least of rank $4$. That is, $\rank \liek \geq 4$. On the other hand, $\sigma_2$ is an outer automorphism by assumption and on Lie 
	algebras of type $\liealg{e}_6$, each such automorphism fixes a subalgebra of rank $4$ by \cref{prop-rank-of-fixed-point-set-of-automorphism}. As $\liek \subseteq \liek_2$,
	we thus must have $\rank \liek = 4$, and so $\lies$ is also a maximal torus for $\liek_2$. Because $(G, K_2)$ is isotropy formal by 
	\cite[theorem 1.1]{Goertsches-automorphisms}, also $(G, K)$ then must be isotropy formal by \cref{prop-formality-depends-on-torus}.
	
	\subsection*{Lie groups of type $\liealg{a}_n$}
	
	Next, we assume that $\lieg$ is of type $\liealg{a}_n$, $n \geq 1$. Fix a notion of positivity $\Delta^{+}$ on the set of all $\lieg$-roots $\Delta$. Passing to conjugates if necessary, 
	we know from \cref{thm-normal-form-automorphisms} that we may express $\sigma_2$ as $\sigma_2 = c_g\circ\nu$, where each element in $\Omega$ is fixed by $\nu$ and where 
	$\nu(\Delta^{+}) = \Delta^{+}$. Hence, if $\Pi \subseteq \Delta^{+}$ denotes the set of simple roots, then $\nu$ permutes $\Pi$ and $\nu|_{\Pi}\colon\Pi\rightarrow\Pi$ is an 
	automorphism of the Dynkin diagram of $\lieg$. As is known, there is just one such automorphism $\Pi\rightarrow\Pi$ apart from the identity, and so by assumption $\nu|_{\Pi}$ 
	must be this non-trivial automorphism. In particular, writing $\ell := \lceil n/2\rceil$, we see that there is an enumeration $\Pi = \{\alpha_1, \ldots, \alpha_n\}$ of the simple roots 
	such that $\nu(\alpha_i) = \alpha_{n-i+1}$ for all $i \leq \ell$. The roots fixed by $\nu$ are then the roots
	\begin{align*}
		\delta_t &:= \alpha_t + \alpha_{t+1} + \ldots + \alpha_{n-(t+1)+1} + \alpha_{n-t+1},
	\end{align*}
	where again $t = 1, \ldots, \ell$. We also observe that the Weyl group element $w = s_{\alpha_t}s_{\alpha_{n-t+1}}$ satisfies $w(\delta_{t+1}) = \delta_t$ and that $w$ commutes 
	with $\nu$. Hence, we may additionally assume that 
	\begin{align*}
		\Omega^{+} &= \{\delta_1, \delta_2, \ldots, \delta_m\}
	\end{align*}
	for some $m \geq 0$. Since $\liet^{\nu} = \bigoplus_{j=1}^{\ell} \CC H_{\delta_j}$ and $\Ad_g|_{\liet} = \prod_{\alpha\in\Omega^{+}} s_{\alpha}$, we conclude 
	that $\lies = \liet^{\sigma_2}$ has as basis the vectors $H_{\delta_{m+1}}, \ldots, H_{\delta_{\ell}}$. However, since $\nu$ is induced by an automorphism of the Dynkin 
	diagram of $\lieg$, we know that the vectors $\beta_1, \ldots, \beta_{\ell}$ with
	\begin{align*}
		\beta_i &:= (\alpha_i)|_{\liet^{\nu}}, 
	\end{align*}
	form a basis for $(\liet^{\nu})^{\ast}$. In fact, since $\nu$ permutes $\Pi$, $\Delta^{+}$ induces a notion of positivity on the $\liet^{\nu}$-roots of the Lie algebra $\lieh := \lieg^{\nu}$
	via restriction, and with respect to this notion of positivity the vectors $\beta_1, \ldots, \beta_{\ell}$ form a simple system. Moreover, $\lieh$ is of type $\liealg{b}_{\ell}$ or $\liealg{c}_{\ell}$, depending on the parity of $n$, and $\beta_1, \ldots, \beta_{\ell-1}$ all have the same length. The 
	element $K_{\beta_i}$ which is dual to $\beta_i$ with respect to $\langle\cdot,\cdot\rangle|_{\lieh\times\lieh}$ then is just $1/2(H_{\alpha_i} + H_{\nu(\alpha_i)})$, 
	and so we see that 
	\begin{align*}
		\lies &= \bigoplus_{j=m+1}^{\ell} \CC H_{\delta_j} = \bigoplus_{j=m+1}^{\ell} \CC K_{\beta_j}.
	\end{align*}
	Thus, $\lies$ is a maximal torus for the subalgebra $\liealg{f}$ of $\lieh$ generated by the root spaces $\lieh_{\beta_j}$ with $j = m + 1, \ldots, \ell$. Since this subalgebra 
	$\liealg{f}$ is totally non-homologous to zero in $\lieh$ (see e.\,g.\,\cite[corollary II.5.7]{thesis}) and $\lieh$ in turn is totally non-homologous to zero in $\lieg$ 
	by \cref{prop-dynkin-diagram-auto-tnhz}, it follows that $\liealg{f}$ is totally non-homologous to 
	zero in $\lieg$ as well. Thus, the isotropy action of $K$ on $G/K$ is equivariantly formal by \cref{prop-formality-depends-on-torus} and \cref{prop-tnhz-implies-isotropy-formality}.
	
	\subsection*{Lie groups of type $\liealg{d}_r$} The last possibility to discuss is when $\lieg$ is of type $\liealg{d}_r$, $r \geq 4$. 	
	We first treat the case that $\sigma_2$ represents an element of order $3$ in $\Out(G)$. This is only possible if
	$r = 4$ \cite[theorem X.3.29]{Helgason}. By \cref{prop-rank-of-fixed-point-set-of-automorphism}, the rank of $\liek_2$ equals $2$. Therefore, $\rank \liek$ either equals $2$ or 
	$1$, because every finite-order automorphism on a non-Abelian, compact Lie group has a nontrivial fixed point set \cite[lemma X.5.3]{Helgason}. If
	$\dim \lies = 2$, then $\lies$ also is a maximal torus for $\liek_2$ and $(G, K)$ is isotropy formal by \cref{prop-formality-depends-on-torus}, since $(G, K_2)$ is isotropy formal
	\cite[theorem 1.1]{Goertsches-automorphisms}. Now assume that $\lies$ is $1$-dimensional. According to \cref{prop-classification-of-involutions} the isomorphism type of $\liek_1$ can only be 
	$\liealg{a}_1\oplus\liealg{a}_1\oplus\liealg{a}_1\oplus\liealg{a}_1$. The only way that $\sigma_2$ fixes a subalgebra of rank $1$ on $\liek_1$ thus is that $\sigma_2$ 
	permutes the simple ideals of $\liek_1$, and so $\liek$ must be of type $\liealg{a}_1$. But every simple rank-$1$ Lie subgroup of a compact Lie group is 
	totally non-homologous to zero (cf.\,\cite[lemma II.5.11]{thesis}), whence the isotropy action is equivariantly formal in this case as well.
	
	Henceforth, we assume that $\sigma_2$ is of order $2$ in $\Out(G)$. As in the case of Lie groups of type $\liealg{a}_n$
	we will have to determine the set $\Omega$, up to conjugation. Because $\sigma_1$, $\sigma_2$ satisfy $\star$ by \cref{thm-reflection} and $\sigma_1$
	is an inner automorphism, we find  a notion of positivity $\Delta^{+}$ on the set of all $G$-roots $\Delta$ as in 
	\cref{cor-inner-automorphisms-normal-form-for-roots-vanishing-on-maximal-torus}. Thus, $\sigma_2 = c_g\circ\nu$ for some automorphism $\nu$ which fixes the elements of 
	$\Omega$ and which also preserves the set $\Delta^{+}$. Enumerate the simple roots $\Pi = \{\alpha_1, \ldots, \alpha_r\}$ as in \cref{fig-dynkin-diagram-type-d}, i.\,e. so that in the Dynkin 
	diagram of $\lieg$ the root $\alpha_{r-2}$ corresponds to the node with a triple link, and such that $\nu$ interchanges $\alpha_{r-2}$ with $\alpha_{r-1}$ and fixes all other 
	simple roots.

	We now proceed similarly as in \cite[theorem II.3.9]{thesis} and describe an algorithm to obtain a standard form for $\Omega^{+}$, up to application of a Weyl group element.
	To this end note that, because $\lieg$ is simple, the Weyl group acts transitively on the set of roots \cite[theorem 10.3]{Humphreys-lie-algebras}. In particular, there is some Weyl group 
	element which maps a given root in $\Omega^{+}$ to the highest root of $\liealg{d}_r$. In fact, every root $\alpha \in \Omega^{+}$ is fixed 
	by $\nu$, so if we express $\alpha$ as $\alpha = m_1\alpha_1 + \ldots + m_r\alpha_r$, then $m_{r-1} = m_r$, and hence, to map $\alpha$ to the highest root
	we only need to use Weyl group elements which are compositions of the reflections $s_{\alpha_1}, \ldots, s_{\alpha_{r-2}}$ or the element $s_{\alpha_{r-1}}\circ s_{\alpha_r}$.
	All of these elements commute with $\nu|_{\liet}$, and actually are represented by elements in $\mathrm{Int}(G)$ which commute with $\nu$ by 
	\cref{cor-inner-automorphisms-normal-form-for-roots-vanishing-on-maximal-torus}. Without loss of generality, we may therefore assume that the highest root $\delta$ is 
	contained in $\Omega^{+}$.
	
	Note that there is exactly one simple root which is non-perpendicular to $\delta$, namely $\alpha_2$. Since the elements of
	$\Omega^{+}$ are mutually perpendicular (\cref{prop-strongly-orthogonal}) and since the highest root is dominant, it therefore follows that every element $\alpha \in \Omega^{+}$
	different from $\delta$ must be contained in the $\ZZ_{\geq 0}$-span of $\Pi - \{\alpha_2\}$; that is, either equal to $\alpha_1$ or contained in the $\ZZ_{\geq 0}$-span of 
	$\alpha_3, \ldots, \alpha_r$. But the $\ZZ_{\geq 0}$-span of $\alpha_3, \ldots, \alpha_r$ is (up to identification) the root lattice of the Lie subalgebra $\lieg_{\geq 3}$ of 
	$\lieg$ generated by $\lieg_{\alpha_3}, \ldots, \lieg_{\alpha_r}$, and this is a Lie algebra of type $\liealg{d}_{r-2}$, $\liealg{a}_3$, or 
	$\liealg{a}_1\oplus\liealg{a}_1\oplus\liealg{a}_1$, depending on whether $r \geq 6$, $r = 5$, or $r = 4$, respectively. Moreover, the Weyl group of $\lieg_{\geq 3}$ can be identified 
	with the subgroup of the Weyl group of $\lieg$ generated by the simple reflections $s_{\alpha_3}, \ldots, s_{\alpha_r}$, and this subgroup leaves invariant both the highest root 
	and $\alpha_1$. 
	
	Proceeding recursively, we thus obtain the following description of $\Omega^{+}$, up to application of a Weyl group element $w$ which commutes with $\nu$. Define
	for $t = 1, \ldots, r - 2$ the root $\delta_t$ by 
	 \begin{align*}
	 	\delta_t &= \alpha_t + 2(\alpha_{t+1} +\ldots + \alpha_{r-2}) + \alpha_{r-1} + \alpha_r.
	\end{align*}
	This is the highest root of the Lie subalgebra $\lieg_{\geq t}$ generated by $\lieg_{\alpha_t}, \ldots, \lieg_{\alpha_r}$, even when $t = r - 2$, in which case
	$\lieg_{\geq t}$ is of type $\liealg{a}_3$. Then there exists an odd index $m$ with $m \leq r - 2$ and
	a set of indices $I \subseteq \{1, 3, \ldots, m\} \cup \{r - 1, r\}$ with the property that $\Omega^{+}$ is equal to
	\begin{align*}
		\{\alpha_i \,|\, i \in I\} \cup \{\delta_1, \delta_3, \ldots, \delta_m\}.
	\end{align*}
	
	\begin{proposition}\label{prop-normal-form-for-d-series}
		Suppose that $\Omega$ is non-empty. There exists an inner automorphism $w$ commuting with $\nu$ such that $w(\Omega^{+})$ is either
		\begin{enumerate}
			\item
			$\{\delta_1, \alpha_1, \delta_3, \alpha_3, \ldots, \delta_m, \alpha_m\}$, where $m$ is odd and $m \leq r - 2$, or
			
			\item
			$\{\delta_1, \delta_3, \ldots, \delta_k\}$, where $k$ is the maximal odd integer with $k \leq r - 2$.
		\end{enumerate}
	\end{proposition}
	\begin{proof}
		We adopt the notation of the paragraph preceding this proposition and assume, without loss of generality, that $w = \id$. First note that in the description of
		$\Omega^{+}$ obtained earlier we must have $I \subseteq \{1, \ldots, m\}$, because $\alpha_{r-1}$ and $\alpha_r$ are not fixed by $\nu$. Next, recall
		from \cref{cor-inner-automorphisms-normal-form-for-roots-vanishing-on-maximal-torus} that the integer
		\begin{align*}
			p(\beta) &:= \sum_{\alpha\in\Omega^{+}}2\frac{\langle\alpha,\beta\rangle}{\langle\alpha,\alpha\rangle}
		\end{align*}
		must be even whenever $\beta$ is equal to one of the simple roots $\alpha_1, \ldots, \alpha_{r-2}$. However, if $\beta = \alpha_{2j}$, then each
		summand in the expression above equals $\pm 1$. Hence, if $2j \leq r - 2$, then
		there must be an even number of roots in $\Omega^{+}$ which are non-perpendicular to $\alpha_{2j}$. 
		
		It follows that either $I$ is empty or that $I = \{1, \ldots, m\}$. Indeed, suppose that $\ell \in I$ and $\ell < m$. This means that $\delta_{\ell}$,
		$\delta_{\ell + 2}$, and $\alpha_{\ell}$ are contained in $\Omega^{+}$. Each of these roots is non-perpendicular to $\alpha_{\ell + 1}$, whence there
		must be at least one more root in $\Omega^{+}$ which is non-perpendicular to $\alpha_{\ell + 1}$. But the only roots of the form $\delta_k$ or $\alpha_k$
		with odd $k$ which are non-perpendicular to $\alpha_{\ell + 1}$ are exactly $\delta_{\ell}$, $\delta_{\ell + 2}$, $\alpha_{\ell}$, and $\alpha_{\ell + 2}$,
		whence $\alpha_{\ell + 2}$ must be contained in $\Omega^{+}$ as well. A similar argument shows that $\alpha_{\ell - 2} \in \Omega^{+}$ if $\ell > 1$.

		To finish the proof it only remains to note that if $I$ is empty, then the integer $m$ must be maximal with $m \leq r - 2$ for otherwise $p(\alpha_{m+1})$ would not be even.		
	\end{proof}
	
	With these results at hand, we can finally prove that the isotropy action is equivariantly formal in case that $\lieg$ is of type $\liealg{d}_r$.
	After conjugation with an inner automorphism, we can assume that $\Omega^{+}$ has one of the two forms given in \cref{prop-normal-form-for-d-series}.  If $\Omega^{+}$ is
	of the first form we have $p(\alpha) = 0$ for all simple roots $\alpha$. Hence, in this case, by representing each factor in $\Ad_g|_{\liet} = \prod_{\alpha\in\Omega^{+}}s_{\alpha}$ 
	as in \cref{prop-simple-root-reflections} we see that $\Ad_g|_{\liet}$ can be lifted to an involution in $G$ (\cref{cor-simple-reflections-on-weight-spaces}), and 
	this choice of lift commutes with $\nu$. Then $\sigma_2$ is an involution and $\liek_1 \cap \liek_2$ is the Lie algebra of a pair of 
	commuting involutions, whence the isotropy action is equivariantly formal by \cref{prop-formality-depends-on-torus} and equivariant formality of
	$\ZZ_2\times\ZZ_2$-symmetric spaces (cf.\,\cite{AmannKollross} or \cite[theorem II.1.2]{thesis}). Therefore, we suppose that
	\begin{align*}
		\Omega^{+} &= \{\delta_1, \delta_3, \ldots, \delta_k\},
	\end{align*}
	where $k$ is the maximal odd integer with $k \leq r - 2$. Now let $\sigma_1 = c_t$ with $t = \exp(U)$ for some element $U \in \liet$. Using that $\sigma_2|_{\liet}$
	is an involution, we can decompose $U$ as $U = U^{+} + U^{-}$ with $U^{+}$ and $U^{-}$ in the $1$- and $(-1)$-eigenspace of $\sigma_2|_{\liet}$, respectively.
	Just as in the proof of \cref{cor-inner-automorphisms-normal-form-for-roots-vanishing-on-maximal-torus} it follows from the commutativity of $\sigma_1$ and $\sigma_2$
	that $c_{\exp(U^{-})}$ must be an involution which commutes with $\sigma_2$, and the argument given in the proof of \cref{cor-eq-formality-in-good-cases} shows that 
	$S$ still is a maximal torus for the fixed point set of $c_{\exp(U^{-})}$ and $\sigma_2$. Consequently, we may assume that $U = U^{-}$. In fact, we may assume that 
	$U = X + \I\pi\sum_{\alpha\in\Omega^{+}}\frac{1}{\langle\alpha,\alpha\rangle} H_{\alpha}$ with $\nu(X) = -X$, 
	cf.\,the proof of \cref{cor-inner-automorphisms-normal-form-for-roots-vanishing-on-maximal-torus}. 
	
	Note that $X$ cannot be zero, for otherwise it would follow from the involutivity of $\sigma_1$ that 
	$p(\beta)$ is even for all roots $\beta$, which is not the case for $\alpha_{r-2}$ (if $r$ is even) 
	or $\alpha_{r-1}$ (if $r$ is odd). Then a straightforward computation using that $X \in \CC(H_{\alpha_{r-1}}-H_{\alpha_r})$ shows that for some $t \in \{r -1, r\}$ and every 
	root $\alpha$ of the form $\alpha = m_1\alpha_1 + \ldots + m_r\alpha_r$ we have
	\begin{align*}
			\sigma_1|_{\lieg_\alpha} &= 
			\begin{cases}
				(-1)^{m_{r-2} + m_t}, 	&\text{$r$ even},	\\
				(-1)^{m_t},			&\text{$r$ odd.}
			\end{cases}
	\end{align*}
	Without loss of generality, assume that $t = r$ and observe that $r$ cannot be even. Indeed, suppose that $r$ is even for a contradiction. A root $\alpha$ has
	$m_{r-2} \in \{0, 1, 2\}$, and if $m_{r-2} = 2$, then necessarily $m_{r-1} = m_r = 1$. The above description of $\sigma_1$ hence shows that
	$\liek_1$ equals $\liet\oplus\bigoplus_{\alpha} \lieg_{\alpha}$, where $\alpha$ ranges over all roots $\alpha$ which are contained in the $\ZZ$-span
	of the the set  $\Phi = \{\alpha_1, \ldots, \alpha_{r-3}\} \cup \{\alpha_{r-2} + \alpha_r, \alpha_{r-1}\}$. In particular, $Z(\liek_1) = \bigcap_{\alpha\in\Phi} \ker(\alpha)$ is $1$-dimensional.
	However, $Z(\liek_1)$ is a $\sigma_2$-invariant subspace, and if $r$ is even, then $\sigma_2(\alpha_{r-1}) = \alpha_r$, whence  
	$Z(\liek_1)$ is also contained in $\bigcap_{j=1}^r \ker(\alpha_j)$. But the latter space is trivial, a contradiction.

	Consequently, we can assume that $r$ is odd, say with $r = 2q + 1$, and $k = 2q - 1$. Because any root $\alpha$ in $\lieg$ has $m_{r-1}, m_r \in \{0, 1\}$, it follows that 
	$\liek_1$ is equal to $\liet + \lieg'$, where $\lieg'$ is the Lie subalgebra generated by the weight spaces of the roots $\alpha_1, \ldots, \alpha_{r-1}$.
	In fact, $\lieg'$ is the semi-simple part of $\liek_1$ and $Z(\liek_1) = \ker(\alpha_1) \cap \ldots \cap \ker(\alpha_{r-1})$. Since $\dim\lies = \liet^{\nu} - q$ by our
	explicit description of $\sigma_2$ and since $\nu$ swaps 
	$\alpha_{r-1}$ with $\alpha_r$, we furthermore see that $\liek$ is of rank $2q - q = q$, which also is the minimal rank of any non-trivial automorphism on a Lie 
	algebra of type $\liealg{a}_{2q}$ (\cref{prop-rank-of-fixed-point-set-of-automorphism}). Hence, since $\lieg'$ is of type $\liealg{a}_{2q}$, we conclude that 
	$\sigma_2|_{\lieg'}$ is an outer automorphism and that $\liek$ is equal to the fixed point set of $\sigma_2$ on $\lieg'$. In particular, $\lies$ is a maximal torus of the fixed point 
	subalgebra of $\sigma_2|_{\lieg'}$ and $\liealg{u} = Z_{\lieg'}(\lies)$ is a maximal torus in $\lieg'$. Thus, since $\lies$ is fixed pointwise by $\sigma_2$, we can define a notion of 
	positivity on the roots of $\lieg'$ (with respect to $\liealg{u}$) which is preserved by $\sigma_2|_{\lieg'}$, and then extend the automorphism of the Dynkin diagram induced by 
	$\sigma_2|_{\liealg{u}}$ to an automorphism $\tau\colon\lieg'\rightarrow\lieg'$. Its fixed point set $\liealg{f} = (\lieg')^{\tau}$ shares the 
	maximal torus $\lies$ with $\liek$. Moreover, $\liealg{f}$ is totally non-homologous to zero in $\lieg'$ by \cref{prop-dynkin-diagram-auto-tnhz}. In fact, if $P_{\lieg'}$ and $P_{\liealg{f}}$ denote, respectively, the 
	polynomials on $\lieg'$ and $\liealg{f}$ which are invariant under the adjoint representation, and if $I \subseteq P_{\lieg'}$ is the ideal generated by all polynomials of odd degree, 
	then restriction of polynomials induces a surjection $P_{\lieg'}/I \rightarrow P_{\liealg{f}}$, because $P_{\liealg{f}}$ is generated by polynomials of even degree 
	\cite[table I, section 3.7]{Humphreys-reflection-groups}. However, by 
	\cite[corollary II.5.10]{thesis} restriction also induces a surjection $P_{\lieg} \rightarrow P_{\lieg'}/I$, $P_{\lieg}$ the invariant poloynomials on $\lieg$, and so 
	$P_{\lieg} \rightarrow P_{\liealg{f}}$ too must be a surjection. That is to say, $\liealg{f}$ is totally non-homologous to zero in $\lieg$. In particular, the pair $(G, F)$ is isotropy formal 
	by \cref{prop-tnhz-implies-isotropy-formality}. Since $\lies$ is a maximal torus for both $\liealg{f}$ and $\liek$, also the pair $(G, K)$ must be isotropy formal 
	(\cref{prop-formality-depends-on-torus}).

	\bibliographystyle{aomalpha}
	\bibliography{gamma}
\end{document}